\documentclass[a4paper, 10pt]{amsart}

\usepackage[T1]{fontenc}
\usepackage[english]{babel}
\usepackage{amsfonts}

\newtheorem{theo}{Theorem}[section]
\newtheorem{prop}[theo]{Proposition}
\newtheorem{defi}[theo]{Definition}
\newtheorem{lemm}[theo]{Lemma}
\newtheorem{coro}[theo]{Corollary}
\newtheorem{conj}[theo]{Conjecture}
\newtheorem{rema}[theo]{Remark}

\newcommand{\mb}{\mathbb}
\newcommand{\mc}{\mathcal}
\newcommand{\mf}{\mathfrak}
\newcommand{\mbf}{\mathbf}
\newcommand{\ra}{\rightarrow}

\DeclareMathOperator{\Ext}{Ext}
\DeclareMathOperator{\Sym}{Sym}
\DeclareMathOperator{\gr}{gr}
\DeclareMathOperator{\Hom}{Hom}
\DeclareMathOperator{\Isom}{Isom}
\DeclareMathOperator{\End}{End}

\DeclareMathOperator{\Spf}{Spf}
\DeclareMathOperator{\Gal}{Gal}

\DeclareMathOperator{\Frob}{Frob}
\DeclareMathOperator{\Ind}{Ind}
\DeclareMathOperator{\Mor}{Mor}
\DeclareMathOperator{\GL}{GL}
\DeclareMathOperator{\SL}{SL}
\DeclareMathOperator{\Ker}{Ker}

\DeclareMathOperator{\Rep}{Rep}

\DeclareMathOperator{\Supp}{Supp}
\DeclareMathOperator{\St}{St}
\DeclareMathOperator{\Sp}{Sp}
\DeclareMathOperator{\sgn}{sgn}

\usepackage{amssymb}
\usepackage{amsmath}

\usepackage[latin1]{inputenc}
\usepackage[all]{xy}

\title{On mod $p$ non-abelian Lubin-Tate theory for $\GL_2(\mb{Q}_p)$}
\date {19th March 2013}
\author{Przemys\l aw Chojecki }
\email{chojecki@math.jussieu.fr}

\begin{document}

\begin{abstract} We analyse the mod p \'etale cohomology of the Lubin-Tate tower both with compact support and without support. We prove that there are no supersingular representations in the $H^1_c$ of the Lubin-Tate tower. On the other hand, we show that in $H^1$ of the Lubin-Tate tower appears the mod p local Langlands correspondence and the mod p local Jacquet-Langlands correspondence, which we define in the text. We discuss the local-global compatibility part of the Buzzard-Diamond-Jarvis conjecture which appears naturally in this context.
\end{abstract}

\maketitle
\tableofcontents

\section{Introduction}

In recent years, the mod p and p-adic local Langlands correspondences emerged as a form of refinement of the l-adic Langlands correspondence for $l = p$. This program was basically started by Christophe Breuil and then, by the work of many people the p-adic Local Langlands correspondence was established for $\GL_2(\mb{Q}_p)$ (see \cite{co} for a final step). Unfortunately, as for now, it is hard to predict how the conjectures should look like for $\GL_2(F)$ where $F$ is a finite extension of $\mb{Q}_p$ (or other groups) basically because there are too many objects on the automorphic side (see \cite{bp}), so that the pure representation theory does not indicate which representations of $\GL_2$ we should choose for our correspondence. A possible remedy to this situation might come by looking at the mod p and completed p-adic cohomology of the Lubin-Tate tower. Let us remind the reader that the classical Local Langlands correspondence was also firstly proved for $\GL_2$ by representation-theoretic methods. Only afterwards by using geometric arguments and finding the correspondence in the l-adic cohomology of the Lubin-Tate tower, it was proved for $\GL_n$ by Harris and Taylor in \cite{ht}. Our aim is to do a step in this direction in the $l=p$ setting, hoping that this will give us an insight how to define a correspondence for other groups than $\GL_2(\mb{Q}_p)$.

By the recent work of Emerton (see \cite{em2}) we know that the p-adic completed (resp. mod p) cohomology of the tower of modular curves realizes the p-adic (resp. mod p) Local Langlands correspondence. In this article we will obtain an analogous but weaker result for the mod p cohomology of the Lubin-Tate tower over $\mb{Q}_p$. In fact, we will analyse both the cohomology with compact support and the cohomology without support of the Lubin-Tate tower. Here are the two main results which we prove:
\newline
\newline
(1) In the first cohomology group $H^1 _{LT, \bar{\mb{F}}_p}$  of the Lubin-Tate tower appears the mod p local Langlands correspondence and the naive mod p Jacquet-Langlands correspondence, meaning that there is an injection of representations
$$ \pi \otimes \bar{\rho} \hookrightarrow H^1 _{LT, \bar{\mb{F}}_p}$$
and $\sigma \otimes \pi \otimes \bar{\rho}$ appears as a subquotient in $H^1 _{LT, \bar{\mb{F}}_p}$, where $\pi$ is a supersingular representation of $\GL_2(\mb{Q}_p)$, $\bar{\rho}$ is its associated local mod p Galois representation and $\sigma$ is the naive mod p Jacquet-Langlands correspondence (for details, see Section 8).
\newline
\newline
(2) The first cohomology group $H^1 _{LT, c, \bar{\mb{F}}_p}$ with compact support of the Lubin-Tate tower does not contain any supersingular representations. This suprising result shows that the mod p situation is much different from its mod l analogue. It also permits us to show that the mod p local Langlands correspondence appears in $H^1$ of the ordinary locus - again a fact which is different from the l-adic setting for supercuspidal representations.
\newline
\newline
Before sketching how we obtain the above results, let us outline the first main difference with the non-abelian Lubin-Tate theory in the l-adic case. When $l\not = p$ the comparison between the Lubin-Tate tower and the modular curve tower is made via vanishing cycles. For that, we need to know that the stalks of vanishing cycles gives the cohomology of the Lubin-Tate tower, or in other words we need an analogue of the theorem proved by Berkovich in \cite{ber2}. But when $l=p$, the statement does not hold anymore (see Remark 3.8.(iv) in \cite{ber2}) and hence we cannot imitate directly the arguments from the l-adic theory.

To circumvent this difficulty, we work from the beginning at the rigid-analytic level and consider embeddings from the ordinary and the supersingular tubes into modular curves. This gives two long exact sequences of cohomology, depending on whether we take compact support or a support in the ordinary locus and we start our analysis by resuming facts about the geometry of modular curves. We recall a decomposition of the ordinary locus, which proves that its cohomology is induced from some proper parabolic subgroup of $\GL_2$. We use this fact several times in order to have vanishing of the cohomology of ordinary locus after localising at a supersingular representation $\pi$ of $\GL_2(\mb{Q}_p)$. We then recall standard facts about admissible representations and review the functor of localisation at $\pi$ which comes out of the work of Paskunas.

We then turn to the analysis of the supersingular locus. In this context, naturally appears a quaternion algebra $D^{\times} /\mb{Q}$ which is ramified exactly at $p$ and $\infty$. We define the local fundamental representation of Deligne in our setting (which appeared for the first time in the letter of Deligne \cite{de}) and we show a decomposition of the cohomology of supersingular locus. At this point we will be able to show that $H^1$ of the tower of modular curves injects into $H^1$ of the Lubin-Tate tower hence proving part of (1). 

Having established this result, we start analysing mod p representations of the p-adic quaternion algebra and define a candidate for the mod p Jacquet-Langlands correspondence $\sigma _{\mf{m}}$ which we later show to appear in the cohomology. It will a priori depend on a global input, namely a maximal ideal $\mf{m}$ of a Hecke algebra corresponding to some modular mod p Galois representation $\bar{\rho}$, but we conjecture that it is independent of $\mf{m}$. This is reasonable as it would follow from the local-global compatibility part of the Buzzard-Diamond-Jarvis conjecture. After further analysis of $\sigma _{\mf{m}}$ we are able to finish the proof of (1).

Using similar techniques, we start analysing the cohomology with compact support of the Lubin-Tate tower. By using the Hochschild-Serre spectral sequence, we are able to reduce (2) to the question of whether supersingular Hecke modules of the mod p Hecke algebra at the pro-p Iwahori level appear in the $H^1_c$ of the Lubin-Tate tower at the pro-p Iwahori level. We solve this question by explicitely computing some cohomology groups.

While proving the above theorems, we will also prove that the first cohomology group of the Lubin-Tate tower and the first cohomology group of the ordinary locus are non-admissible smooth representations. In particular, they are much harder to describe than their mod l analogues. Moreover our model for the mod p Jacquet-Langlands correspondence $\sigma _{\mf{m}}$ (actually we propose three candidates for the correspondence which we discuss in Section 7.4) is a representation of $D^{\times}(\mb{Q}_p)$ of infinite length. This indicates that already for $D^{\times}(\mb{Q}_p)$ the mod p Langlands correspondence is complicated (as in the work of \cite{bp}, representations in question are not of finite length). On the other hand, the case of $D^{\times}(\mb{Q}_p)$ is much simpler than that of $\GL_2(F)$ for $F$ a finite extension of $\mb{Q}_p$, and hence we might be able to describe $\sigma _{\mf{m}}$ precisely. Natural question in this discussion is the local-global compatibility part of the Buzzard-Diamond-Jarvis conjecture (see Conjecture 4.7 in \cite{bdj}) which says that we have an isomorphism
$$\mbf{F}[\mf{m}] \simeq \sigma _{\mf{m}} \otimes \pi^p(\bar{\rho})$$
where $\mbf{F}$ denotes locally constant functions on $D^{\times}(\mb{Q}) \backslash D^{\times}(\mb{A}_f)$ with values in $\bar{\mb{F}}_p$ and $\pi^p(\bar{\rho})$ is a representation of $\GL_2(\mb{A}_f ^p)$ associated to $\bar{\rho}$ by the modified Langlands correspondence.

At the end, we remark that our arguments work well in the $l\not =p$ setting and omit the use of vanishing cycles. As some of our arguments are geometric, we can also get similar results in the p-adic setting. We hope to return to this issue in our future work. Also, the geometry of modular curves is very similar to the geometry of Shimura curves and hence we hope that some of the reasonings in this article will give an insight into the nature of the mod p local Langlands correspondence of $\GL_2(F)$ for $F$ a finite extension of $\mb{Q}_p$.
\newline 
\newline
\textbf{Acknowledgements.} I would like to heartily thank my advisor Jean-Francois Dat who proposed this problem to me and who patiently answered many of my questions which aroused during the work on this subject. I thank Michael Rapoport for inviting me to give two talks on the content of this paper at the University of Bonn in January 2013 and Peter Scholze for asking useful questions. I also thank Yongquan Hu for a useful remark.

\section{Geometry of modular curves}

Let $X(Np^m)$ be the Katz-Mazur compactification of the modular curve associated to the moduli problem $(\Gamma (p^n),\Gamma _1(N))$ (see \cite{km}) which is defined over $\mb{Z}[1/N, \zeta _{p^n}]$, where $\zeta _{p^n}$ is a primitive $p^n$-th root of unity, that is $X(Np^m)$ parametrizes (up to isomorphism) triples $(E, \phi, \alpha)$, where $E$ is an elliptic curve, $\phi: (\mb{Z} / p^n \mb{Z})^2 \ra E[p^n]$ is a Drinfeld level structure and $\alpha : \mb{Z} / N \mb{Z} \ra E[N]$ is a $\Gamma _1 (N)$-structure. We consider the integral model of it defined over $\mb{Z} _p ^{nr}[\zeta _{p^n}]$, where $\mb{Z} _p ^{nr}$ is the maximal unramified extension of $\mb{Z}_p$, which we will denote also by $X(Np^m)$. Let us denote by $X(Np ^m) ^{an}$ the analytification of $X(Np^m)$ which is a Berkovich space. 

Recall that there exists a reduction map $\pi : X(Np^m)^{an} \ra \overline{X(Np ^m)}$, where $\overline{X(Np ^m)}$ is the special fiber of $X(Np^m)$. We define $\overline{X(Np ^m) _{ss}}$ (resp. $\overline{X(Np ^m) _{ord}}$) to be the set of supersingular (resp. ordinary) points in $\overline{X(Np ^m)}$. Define the tubes $X(Np^m) _{ss} = \pi ^{-1} (\overline{X(Np ^m) _{ss}})$ and $X(Np^m) _{ord} = \pi ^{-1} (\overline{X(Np ^m) _{ord}})$ inside $X(Np ^m) ^{an}$ of supersingular and ordinary points respectively. 

\subsection{Two exact sequences}

We know that $X(Np^m) _{ss}$ is an open analytic subspace of $X(Np^m) ^{an}$ isomorphic to some copies of Lubin-Tate spaces, where the number of copies is equal to the number of points in $\overline{X(Np ^m) _{ss}}$ (see section 3 of \cite{bu}).  We have a decomposition $X(Np ^m) ^{an} = X(Np ^m) _{ss} \cup X(Np ^m) _{ord}$ and we put $j: X(Np ^m) _{ss} \hookrightarrow X(Np ^m) ^{an}$ and $i: X(Np ^m) _{ord} \ra X(Np ^m) ^{an}$. We remark that $j$ is an open immersion of Berkovich spaces. Let $F$ be a sheaf in the \'etale topoi of $X(Np^m) ^{an}$. By the general formalism of six operations (due in this setting to Berkovich, see \cite{ber4}) we have a short exact sequence:
$$ 0 \ra j _{!} j^{*} F \ra F \ra i_{*} i ^{*} F \ra 0$$
which gives a long exact sequence of \'etale cohomology groups:
$$... \ra H^0 (X(Np^m) _{ord}, i^* F) \ra H^1 _{c} (X(Np^m) _{ss}, F) \ra H^1  (X(Np^m) ^{an}, F) \ra H^1  (X(Np^m) _{ord}, i^* F) \ra ...$$

On the other hand, we can consider a similar exact sequence for the cohomology without compact support, but instead considering support on the ordinary locus. This results in the long exact sequence
$$... \ra H^1 _{X_{ord}} (X(Np^m)^{an}, F) \ra H^1  (X(Np^m) ^{an}, F) \ra H^1  (X(Np^m) _{ss}, i^* F) \ra H^2 _{X_{ord}}(X(Np^m)^{an}, F) \ra ...$$
where we have denoted by $H^1 _{X_{ord}} (X(Np^m)^{an}, F)$ the \'etale cohomology of $X(Np^m)^{an}$ with support on $X(Np^m)_{ord}$. Because of the vanishing of the cohomology with compact support of the supersingular locus localised at $\pi$ (see the explanation in the next sections), this exact sequence will be of more importance to us later on. We will analyse those two exact sequences simultanously.

\subsection{Decomposition of ordinary locus}

Let us recall that we have the Weil pairing on elliptic curves
$$ e_{p^m} : E[p^m] \times E[p^m] \ra \mu _{p^m}$$
Denote by $\zeta _{p^m}$ a $p^m$-th primitive root of unity. For $a \in (\mb{Z}/p^m \mb{Z})^*$ we define a substack $X(Np^m)_a$ of $X(Np^m)$ as the moduli problem which classifies elliptic curves $E$ with level structures $(\phi, \alpha)$ such that $\ e_{p^m}(\phi( \begin{smallmatrix}1 \\ 0 \end{smallmatrix}), \phi(\begin{smallmatrix} 0 \\ 1 \end{smallmatrix})) = \zeta _{p^m} ^a$. This moduli problem is representable by a scheme over $\mb{Z}_p ^{nr}[\zeta _{p^m}]$ (see chapter 9 of \cite{km}). Moreover the coproduct $\coprod _a X(Np^m)_a$ is a regular model of $X(Np^m)$ over $\mb{Z}_p^{nr}[\zeta _{p^m}]$.

Let us denote by $\overline{X(Np^m)} _{a, ord}$ the ordinary locus of the reduction of $X(Np^m) _{a}$. We recall (see for example chapter 13 of \cite{km}) that the set of irreducible components of $\overline{X(Np^m)} _{ord}$ consists of smooth curves $C_{a,b}(Np^m)$ defined on points by:
$$C _{a,b}(Np^m)(S) = \{ (E,\phi, \alpha) \in \overline{X(Np^m)} _{a, ord}(S) \ | \ e_{p^m}(\phi( \begin{smallmatrix}1 \\ 0 \end{smallmatrix}), \phi(\begin{smallmatrix} 0 \\ 1 \end{smallmatrix})) = \zeta _{p^m} ^a \textrm{  and  } \Ker \phi = b \}$$
where $a \in (\mb{Z}/p^m \mb{Z})^*$ and $b \in \mb{P} ^1(\mb{Z}/p^m \mb{Z})$ is regarded as a line in $\mb{Z}/p^m \mb{Z} \times \mb{Z}/p^m \mb{Z}$. We observe that $\zeta _{p^m} ^a = 1$ modulo p.

We are interested in lifting $C_{a,b}(Np^m)$ to characteristic zero and so we put 
$$\mb{X} _{a,b}(Np^m) = \pi ^{-1} (C _{a,b}(Np^m))$$
Hence $\{\mb{X} _{a,b}(Np^m) \}$ for $a \in (\mb{Z}/p^m \mb{Z})^*$, $b \in \mb{P} ^1(\mb{Z}/p^m \mb{Z})$ form a decomposition of the ordinary locus $X(Np^m) _{ord}$ because different $C_{a,b}(Np^m)$ intersect only at supersingular points. The spaces $\mb{X} _{a,b}(Np^m)$ may be regarded as analytifications of Igusa curves. For a detailed discussion, see \cite{col}. We do not determine here whether $\mb{X} _{a,b}(Np^m)$ are precisely the connected components of $X(Np^m)_{ord}$. We remark also that one can give a moduli description of each $\mb{X} _{a,b}(Np^m)$. 


There is an action of $\GL_2(\mb{Z}/p^m \mb{Z})$ on $X(Np^m) ^{an}$ which is given on points by:
$$(E,\phi, \alpha) \cdot g = (E, \phi \circ g, \alpha)$$
for $g \in \GL_2(\mb{Z}/p^m \mb{Z})$. Observe that if $e_{p^m}(\phi( \begin{smallmatrix}1 \\ 0 \end{smallmatrix}), \phi(\begin{smallmatrix} 0 \\ 1 \end{smallmatrix})) = \zeta _{p^m} ^a$, then $e_{p^m}((\phi \circ g)( \begin{smallmatrix}1 \\ 0 \end{smallmatrix}), (\phi \circ g)(\begin{smallmatrix} 0 \\ 1 \end{smallmatrix})) = \zeta _{p^m} ^{a \cdot \textrm{det} g}$ for $g \in \GL_2(\mb{Z}/p^m \mb{Z})$ and so $g$ induces an isomorphism between $\mb{X} _{a,b}(Np^m)$ and $\mb{X} _{a \cdot \textrm{det} g, g^{-1} \cdot b}(Np^m)$.

For $b \in \mb{P} ^1(\mb{Z}/p^m \mb{Z})$ there is a Borel subgroup $B_m(b)$ in $\GL_2(\mb{Z}/p^m \mb{Z})$  which fixes $b$ and hence the Borel subgroup $B_m(b)^+ = B_m(b) \cap \SL _2(\mb{Z}/p^m \mb{Z})$ in $\SL _2(\mb{Z}/p^m \mb{Z})$ stabilises $\mb{X} _{a,b}(Np^m)$.

Let $b = \infty = (\begin{smallmatrix}1 \\ 0 \end{smallmatrix}) \in \mb{P} ^1(\mb{Z}/p^m \mb{Z})$. By the above considerations we have
$$H^i (X(Np^m) _{ord}, i^* F) = \bigoplus _{a,b} H^i  (\mb{X} _{a,b}(Np^m), (i^* F) _{| \mb{X} _{a,b}(Np^m)}) \simeq $$
$$\simeq \Ind ^{\GL_2(\mb{Z}/p^m \mb{Z})} _{B_m(\infty)}\left( \bigoplus _{a} H^i  (\mb{X} _{a,\infty}(Np^m), (i^* F) _{| \mb{X} _{a,\infty}})) \right)$$
and also
$$ H^1 _{X_{ord}} (X(Np^m)^{an}, F) \simeq \Ind ^{\GL_2(\mb{Z}/p^m \mb{Z})} _{B_m(\infty)} \left( \bigoplus _{a} H^1 _{\mb{X} _{a, \infty}(Np^m)}(X(Np^m)^{an}, F) \right)$$
Those results will be extremely useful for us later on, when we introduce the localisation at a given supersingular representation. 

\subsection{Supersingular points}

Let us denote by $D$ the quaternion algebra over $\mb{Q}$ which is ramified precisely at $p$ and at $\infty$. We recall the description of supersingular points $\overline{X(Np ^m) _{ss}}$ which has appeared in \cite{de} and then was explained in \cite{ca}, sections 9.4 and 10.4. Fix a supersingular elliptic curve $\overline{E}$ over $\mb{F} _p$ and a two-dimensional vector space $V$ over $\mb{Q} _p$. Let $\textrm{det}(\overline{E}) = \mb{Z}$ be the determinant of $\overline{E}$. Denote by $W(\bar{\mb{F}} _p/ \mb{F} _p)$ the Weil group of $\mb{F} _p$ and put
$$\Delta = \left( W(\bar{\mb{F}} _p/ \mb{F} _p) \times \Isom (\textrm{det}(\overline{E})\otimes _{\mb{Z}} \mb{Q} _p(1), \wedge ^2 V) \right) / {\sim}$$
where $\sim$ is defined by $(\sigma, \beta) \sim (\sigma \Frob ^k, p^{-k} \beta)$ for $k \in \mb{Z}$, where $\Frob: x \mapsto x^p$ is a Frobenius map. We define $K_m$ to be the kernel of  $D^{\times}(\mb{Z} _p) \ra D^{\times}( \mb{Z} / p^m \mb{Z})$ and we let $K(N) = \{ (\begin{smallmatrix} a & b \\ c & d \end{smallmatrix}) \in \GL_2(\mb{Z}) \ | \ a \equiv 1 \mod N \textrm{  and  } c \equiv 0 \mod N \}$, viewed as a subgroup of $\GL_2(\mb{A} _f ^p)$ by the diagonal embedding. Then:
$$\overline{X(Np ^m) _{ss}} = \Delta / K_m \times _{D^{\times}(\mb{Q})} \GL_2(\mb{A} _f ^p) / K(N)$$
Every $\delta \in \Delta / K_m$ furnishes a supersingular elliptic curve $E(\delta)$, so that for every $\delta \in \Delta$ we can consider the Lubin-Tate tower $LT_{\delta} = \varprojlim _m LT_{\delta}(p^m)$, which is the generic fiber of the deformation space of the formal group attached to $E(\delta)$ and where $LT _{\delta}(p^m)$ denotes the generic fiber of the deformation space of formal groups with $p^m$-level structure (see \cite{da} for details on the Lubin-Tate tower). Let us denote by $\mb{E}(\delta)$ the universal formal group deforming the formal group attached to $E(\delta)$ and let $\mb{E}(\Delta) = \coprod _{\delta \in \Delta} \mb{E}(\delta)$.  By 9.4 of \cite{ca}, the universal formal group over $\varprojlim _{N, p^m} \overline{X(Np ^m) _{ss}}$ is isomorphic to $\mb{E}(\Delta) \times _{D^{\times}(\mb{Q})} \GL_2(\mb{A} _f ^p)$ and hence we conclude that 
$$\varprojlim _{Np^m} X(Np ^m) _{ss} \simeq LT _{\Delta} \times _{D^{\times}(\mb{Q})} \GL_2(\mb{A} _f ^p)$$
where $LT _{\Delta} = \coprod _{\delta \in \Delta} LT _{\delta}$. We also get a description at a finite level
$$X(Np ^m) _{ss} \simeq LT _{\Delta / K_m} \times _{D^{\times}(\mb{Q})} \GL_2(\mb{A} _f ^p) / K(N)$$
where $LT _{\Delta / K_m} = \coprod _{\delta \in \Delta / K_m} LT _{\delta}(p^m)$.

These results will allow us later on to define the local fundamental representation and analyze the action of the quaternion algebra $D^{\times}$.

\section{Admissibility of cohomology groups}

In this section we will recall the notion of admissibility in the context of mod p representations. It will be crucial in our study of cohomology.
\subsection{General facts and definitions} We start with general facts about admissible representations. In our definitions, we will follow \cite{em3}. Let $k$ be a field of characteristic $p$ and let $G$ be a connected reductive group over $\mb{Q}_p$.

\begin{defi} Let $V$ be a representation of $G$ over $k$. A vector $v \in V$ is smooth if $v$ is fixed by some open subgroup of $G$. Let $V_{sm}$ denote the subset of smooth vectors of $V$. We say that a $G$-representation $V$ over $k$ is smooth if $V=V_{sm}$.

A smooth $G$-representation $V$ over $k$ is admissible if $V^H$ is finitely generated over $k$ for every open compact subgroup $H$ of $G$.
\end{defi}
\begin{prop} The category of admissible $k$-representations is abelian. 
\end{prop}
\begin{proof} This category is (anti-)equivalent to the category of finitely generated augmented modules over certain completed group rings. See Proposition 2.2.13 and 2.4.11 in \cite{em3}.
\end{proof}
Now, we will prove an analogue of Lemma 13.2.3 from \cite{bo} in the $l=p$ setting. We will later apply this lemma to the cohomology of the ordinary locus to force its vanishing after localisation at a supersingular representation of $\GL_2(\mb{Q}_p)$.
\begin{lemm}
For any smooth admissible representation $(\pi,V)$ of the parabolic subgroup $P \subset G$ over $k$, the unipotent radical $U$ of $P$ acts trivially on $V$.
\end{lemm}
\begin{proof} Let $L$ be a Levi subgroup of $P$, so that $P=LU$. Let $v \in V$ and let $K_P = K_L K_U$ be a compact open subgroup of $P$ such that $v\in V ^{K_P}$. We choose an element $z$ in the centre of $L$ such that:
$$z^{-n}K_P z^n \subset ... \subset z^{-1}K_P z \subset K_P \subset z K_P z^{-1} \subset ... \subset z^n K_P z^{-n} \subset ...$$
and $\bigcup _{n \geq 0} z^n K_P z^{-n} = K_L U$. For every $n$ and $m$, modules $V^{z^{-n}K_P z^n}$ and $V^{z^{-m}K_P z^m}$ are of the same length as they are isomorphic via $\pi(z ^{n-m})$ and hence we have not only an isomorphism but an equality $V^{z^{-n}K_P z^n} = V^{z^{-m}K_P z^m}$. Thus for every $x \in V ^{K_P}$ we have $x \in V^{K_P} = V^{z^{-n}K_P z^n}=V^{K_L U}$ which is contained in $V^U$.  
\end{proof}
We also record the following result of Emerton for the future use.
\begin{lemm} Let $V = \Ind _P ^G W$ be a parabolic induction. If $V$ is a smooth admissible representation of $G$ over $k$, then $W$ is a smooth admissible representation of $P$ over $k$.
\end{lemm}
\begin{proof} This follows from Theorem 4.4.6 in \cite{em3}.
\end{proof}

\subsection{Cohomology and admissibility}

In \cite{em1}, Emerton has introduced the completed cohomology, which plays a crucial role in the p-adic Langlands program. The most important thing for us right now is the fact that those cohomology groups for modular curves are admissible as $\GL_2(\mb{Q}_p)$-representations. We have

\begin{prop} The $\GL_2(\mb{Q} _p)$-representation 
$$\widehat{H} ^1(X(N), \bar{\mb{F}}_p) = \varinjlim _m H^1  (X(Np^m) ^{an}, \bar{\mb{F}} _p)$$ 
is admissible.
\end{prop}
\begin{proof} This is Theorem 2.1.5 of \cite{em1} (see also Theorem 1.16 in \cite{ce}).
\end{proof}

By formal properties of the category of admissible representations, which form a Serre subcategory of the category of smooth representations (see Proposition 2.2.13 in \cite{em3}), the above result permits us to deduce admissibility for other cohomology groups which are of interest to us. Let us remark that we can define also the cohomology of the Lubin-Tate tower with compact support:
\begin{rema} A priori, cohomology with compact support is a covariant functor. But using the adjunction map
$$\Lambda \ra \pi _* \pi ^* \Lambda \simeq \pi _! \pi ^! \Lambda  $$
where $\Lambda$ is a constant sheaf and $\pi: X(Np^{m+1}) _{ss}  \ra  X(Np^{m}) _{ss}$ is finite (hence $\pi _* = \pi _!$) and \'etale (hence $\pi ^! = \pi ^*$) by the properties of Lubin-Tate tower, we get maps $H^i _c(X(Np^{m}) _{ss}, \Lambda) \ra H^i _c(X(Np^{m+1}) _{ss}, \Lambda)$ compatible with $H^i (X(Np^{m})^{an}, \Lambda) \ra H^i (X(Np^{m+1}) _{ss}, \Lambda)$.
\end{rema}

We start firstly by analysing cohomology groups which appear in the exact sequence for the cohomology with compact support. We have
\begin{prop} The $\GL_2(\mb{Q} _p)$-representation 
$$\widehat{H}^0 (X(N) _{ord}, \bar{\mb{F}}_p) = \varinjlim _m H^0 (X(Np^m) _{ord}, \bar{\mb{F}}_p)$$
is admissible.
\end{prop}
\begin{proof}
The number of connected components of $X(Np^m)_{ord}$ is finite and let $d(Np^m)$ be their number. For $s>0$, we have 
$$H^0(X(Np^m)_{ord}, \bar{\mb{F}}_p) = (\bar{\mb{F}}_p)^{d(Np^m)}$$
hence $\varinjlim _m H^0(X(Np^m)_{ord}, \bar{\mb{F}}_p)$ is admissible.
\end{proof}

We deduce
\begin{prop} The $\GL_2(\mb{Q} _p)$-representation 
$$\widehat{H} ^1_c(X(N)_{ss}, \bar{\mb{F}}_p) = \varinjlim _m H^1 _{c} (X(Np^m) _{ss}, \bar{\mb{F}} _p)$$
is admissible.
\end{prop}
\begin{proof}
We consider the exact sequence from 2.1:
$$ ... \ra \widehat{H}^0 (X(N) _{ord}, \bar{\mb{F}}_p) \ra \widehat{H}^1 _{c} (X(N) _{ss}, \bar{\mb{F}}_p) \ra \widehat{H}^1  (X(N) ^{an}, \bar{\mb{F}}_p) \ra \widehat{H}^1  (X(N) _{ord}, \bar{\mb{F}}_p) \ra ...$$
and we conclude using the fact that admissible representations form a Serre subcategory of smooth representations and the propositions proved above.
\end{proof}
We remark that the cohomology with compact support of the Lubin-Tate tower is much easier to work with than the cohomology without the support. This is because the latter will turn out to be non-admissible.

We finish this section with the following proposition
\begin{prop} The $\GL_2(\mb{Q}_p)$-representation 
$$\widehat{H} ^1 _{X_{ord}} (X(N), \bar{\mb{F}}_p) = \varinjlim _m H^1 _{X_{ord}} (X(Np^m) ^{an}, \bar{\mb{F}} _p)$$
is admissible.
\end{prop}
\begin{proof} This follows from the exact sequence (we use the notations from the previous section)
$$ H^0 (X(Np^m)_{ss}, \bar{\mb{F}}_p) \ra H^1 _{X_{ord}} (X(Np^m)^{an}, \bar{\mb{F}}_p) \ra H^1  (X(Np^m) ^{an}, \bar{\mb{F}}_p)$$
and Proposition 3.5.
\end{proof}

\section{Supersingular representations}

In this section we recall results on the structure of admissible representations and we apply them to the exact sequence of cohomology groups that we have introduced before, getting the first comparison between the cohomology of the Lubin-Tate tower and the cohomology of the tower of modular curves. We will start with a reminder on the mod p local Langlands correspondence. A reader should consult \cite{be2} for references to proofs of cited facts.

\subsection{Mod p local Langlands correspondence} Let $\omega _n$ be the fundamental character of Serre of level $n$ which is defined on inertia group $I$ via $\sigma \mapsto \frac{\sigma (p^{1/n})}{p^{1/n}}$. Let $\omega$ be the mod p cyclotomic character. For $h\in \mb{N}$, we write $\Ind \ \omega _n ^h$ for the unique semisimple $\bar{\mb{F}}_p$-representation of $G_{\mb{Q}_p}$ which has determinant $\omega ^h$ and whose restriction to $I$ is isomorphic to $\omega _{n} ^h \oplus \omega _n ^{ph} \oplus ... \oplus \omega _n ^{p^{n-1}h}$. If $\chi: G_{\mb{Q}_p} \ra k ^{\times}$ is a character, we will denote by $\rho(r,\chi)$ the representation $\Ind (\omega _2 ^{r+1}) \otimes \chi$ which is absolutely irreducible if $r \in \{0,...,p-1\}$. In fact, any absolutely irreducible representation of $G_{\mb{Q}_p}$ of dimension 2 is isomorphic to some $\rho (r,\chi)$ for $r \in \{0,...,p-1\}$. We remark that $\Ind \ \omega _2 ^{r+1}$ is not isomorphic to the induced representation $\Ind _{G_{\mb{Q}_{p^2}}} ^{G_{\mb{Q}_p}} \omega _2 ^{r+1}$, because of the condition which we put on the determinant. In fact, computing the determinant of $\Ind _{G_{\mb{Q}_{p^2}}} ^{G_{\mb{Q}_p}} \omega _2 ^{r+1}$, one sees that  
$$\Ind \ \omega _2 ^{r+1} = \Ind _{G_{\mb{Q}_{p^2}}} ^{G_{\mb{Q}_p}} (\omega _2 ^{r+1} \cdot \sgn )$$
where $\sgn$ is the $\bar{\mb{F}}_p$-character of $G_{\mb{Q}_{p^2}}$ which factors through $\mb{F} _{p^2}^{\times} \times \mb{Z}$ and which is trivial on $\mb{F}_{p^2} ^{\times}$ and takes the Frobenius of $G_{\mb{Q}_{p^2}}$ to $-1$ in $\mb{Z}$ (we have to make a choice of a uniformiser to have the map $G_{\mb{Q}_{p^2}} \twoheadrightarrow \mb{F} _{p^2}^{\times} \times \mb{Z}$).

On the $\GL_2$-side, one considers representations $\Sym ^r k^2 $ inflated to $\GL_2(\mb{Z}_p)$ and then extended to $\GL_2(\mb{Z}_p)\mb{Q}_p ^{\times}$ by making $p$ acts by identity. We then consider the induced representation 
$$\Ind _{\GL_2(\mb{Z}_p)\mb{Q}_p ^{\times}} ^{\GL_2(\mb{Q}_p)} \Sym ^r k^2 $$
One can show that the endomorphism ring (a Hecke algebra) $\End _{k[\GL_2(\mb{Q}_p)]}( \Ind _{\GL_2(\mb{Z}_p)\mb{Q}_p ^{\times}} ^{\GL_2(\mb{Q}_p)} \Sym ^r k^2 )$ is isomorphic to $k[T]$, where $T$ corresponds to the double class $\GL_2(\mb{Z}_p)\mb{Q}_p ^{\times} ( \begin{smallmatrix} p & 1 \\ 0 & 1 \end{smallmatrix} ) \GL_2(\mb{Z}_p) \mb{Q}_p ^{\times}$. For a character $\chi: G_{\mb{Q}_p} \ra k ^{\times}$ and $\lambda \in k$. we introduce representations:
$$\pi(r, \lambda, \chi) = \frac{\Ind _{\GL_2(\mb{Z}_p)\mb{Q}_p ^{\times}} ^{\GL_2(\mb{Q}_p)} \Sym ^r k^2 }{T-\lambda} \otimes (\chi \circ \det)$$ 
For $r \in \{0,...,p-1\}$ such that $(r,\lambda) \not \in \{(0, \pm 1), (p-1, \pm 1)\}$, the representation $\pi(r, \lambda, \chi)$ is irreducible. If $\lambda = \pm 1$, then $\pi(r, \lambda, \chi)$ appears as either a subrepresentation or a subquotient of the special representation $\Sp$. One proves that $\chi \circ \det$, $\Sp \otimes (\chi \circ \det)$ and $\pi(r, \lambda, \chi)$ for $r \in \{0,...,p-1\}$ and $(r,\lambda) \not \in \{(0, \pm 1), (p-1, \pm 1)\}$ are all the smooth irreducible representations of $\GL_2(\mb{Q}_p)$. 

This explicit description gives a mod p correspondence by associating $\rho(r,\chi)$ to $\pi (r, 0 , \chi)$.

\subsection{Supersingular representations}

Let us a fix a supersingular representation $\pi$ of $\GL_2(\mb{Q} _p)$ on a $\bar{\mb{F}} _p$-vector space with a central character $\xi$.  Recall the following result of Paskunas:
\begin{prop} Let $\tau$ be an irreducible smooth representation of $\GL_2(\mb{Q} _p)$ admitting a central character. If $\Ext^1 _{\GL_2(\mb{Q}_p)}(\pi, \tau) \not = 0$ then $\tau \simeq \pi$.
\end{prop} 
\begin{proof} See \cite{pa2} and \cite{pa3} for the case $p=2$.
\end{proof}
This result permits us to conclude that the $\GL_2(\mb{Q}_p)$-block of any supersingular representation consists of one element - the supersingular representation itself. Here, by a $\GL_2(\mb{Q}_p)$-block we mean an equivalence class for a relation defined as follows. We write $\pi \sim \tau$ if there exists a sequence of irreducible smooth admissible $\bar{\mb{F}}_p$-representations of $\GL_2(\mb{Q}_p)$: $\pi_0 = \pi, \pi _1,..., \pi _n = \tau$ such that for each $i$ one of the following conditions holds: 

1) $\pi _i \simeq \pi _{i+1}$,  

2) $\Ext ^1 _{\GL_2 (\mb{Q}_p)} (\pi _i, \pi _{i+1}) \not = 0$,  

3) $\Ext ^1 _{\GL_2 (\mb{Q}_p)} (\pi _{i+1}, \pi _{i}) \not = 0$.

One can find a description of all $\GL_2(\mb{Q}_p)$-blocks in \cite{pa3} or \cite{pa1}. The general result of Gabriel on the block decomposition of locally finite categories gives:
\begin{prop} We have a decomposition:
$$\Rep ^{adm} _{\GL_2(\mb{Q} _p),\xi}(\bar{\mb{F}}_p) = \Rep ^{adm} _{\GL_2(\mb{Q} _p),\xi}(\bar{\mb{F}}_p) _{(\pi)} \oplus \Rep ^{adm} _{\GL_2(\mb{Q} _p),\xi}(\bar{\mb{F}}_p) ^{(\pi)}$$
where $\Rep ^{adm} _{\GL_2(\mb{Q} _p),\xi}(\bar{\mb{F}}_p)$ is the (abelian) category of smooth admissible $\bar{\mb{F}}_p$-representations admitting a central character $\xi$, $\Rep ^{adm} _{\GL_2(\mb{Q} _p),\xi}(\bar{\mb{F}}_p) _{(\pi)}$ (resp. $\Rep ^{adm} _{\GL_2(\mb{Q} _p),\xi}(\bar{\mb{F}}_p) ^{(\pi)}$) is the subcategory of it consisting of representations $\Pi$ whose all the irreducible subquotients are (resp. are not) isomorphic to $\pi$. 
\end{prop}
\begin{proof} See Proposition 5.32 in \cite{pa1}.
\end{proof}
This result permit us to consider the localisation functor with respect to $\pi$
$$V \mapsto V_{(\pi)}$$
on the category of admissible representations such that all irreducible subquotients of $V _{(\pi)}$ are isomorphic to the fixed $\pi$.

\begin{rema} We note that the condition on the existence of central characters is not important. Central characters always exist by the work of Berger (\cite{be}) in the mod p case.
\end{rema}  

\subsection{Cohomology with compact support}
We apply the localisation functor to the three admissible terms in the exact sequence obtained from 2.1:
$$... \ra \widehat{H}^0 (X(N) _{ord}, \bar{\mb{F}}_p) \ra \widehat{H}^1 _{c} (X(N) _{ss}, \bar{\mb{F}}_p) \ra \widehat{H}^1  (X(N) ^{an}, \bar{\mb{F}}_p) \ra \widehat{H}^1  (X(N) _{ord}, \bar{\mb{F}}_p) \ra ...$$
getting the exact sequence
$$  \widehat{H}^0 (X(N) _{ord}, \bar{\mb{F}}_p) _{(\pi)} \ra \widehat{H}^1 _{c} (X(N) _{ss}, \bar{\mb{F}}_p) _{(\pi)} \ra \widehat{H}^1  (X(N) ^{an}, \bar{\mb{F}}_p) _{(\pi)} $$
For $a \in \mb{Z}_p ^{\times}$ let us define in the light of 2.2
$$\widehat{H} ^1 (\mb{X} _{a, \infty}(N), \bar{\mb{F}}_p) = \varinjlim _m H^1 (\mb{X} _{a, \infty}(Np^m), \bar{\mb{F}}_p)$$
Recall now, that after 2.2, $\widehat{H}^0 (X(N) _{ord}, \bar{\mb{F}}_p)$ is an admissible representation isomorphic to the induced representation 
$$\Ind ^{\GL_2(\mb{Q}_p)} _{B_{\infty}(\mb{Q}_p)} \left( \bigoplus _{a} \widehat{H}^0  (\mb{X} _{a,\infty}(N), \bar{\mb{F}}_p) \right) $$
where $B_{\infty}(\mb{Q}_p)$ is the Borel subgroup of upper triangular matrices in $\GL _2(\mb{Q}_p)$ and $a$ goes over $\mb{Z}_p ^{\times}$. On this representation unipotent group acts trivially by lemma 3.3 (which we can use thanks to lemma 3.4) and hence we see that it is induced from the tensor product of characters. This means that after localisation at $\pi$ this representation vanishes
$$\Ind ^{\GL_2(\mb{Q}_p)} _{B_{\infty}(\mb{Q}_p)} \left( \bigoplus _{a} \widehat{H}^0  (\mb{X} _{a,\infty}(N), \bar{\mb{F}}_p) \right) _{(\pi)} = 0$$
and we arrive at
\begin{theo} We have an injection of representations
$$\widehat{H}^1 _{c} (X(N) _{ss}, \bar{\mb{F}}_p)_{(\pi)} \hookrightarrow \widehat{H}^1  (X(N), \bar{\mb{F}}_p)_{(\pi)}$$
\end{theo}

By taking yet another direct limit, we define
$$\widehat{H}^1 _{ss, c, \bar{\mb{F}}_p} = \varinjlim _N \widehat{H}^1 _{c} (X(N) _{ss}, \bar{\mb{F}}_p)$$
$$\widehat{H} ^1 _{\bar{\mb{F}}_p} =\varinjlim _N \widehat{H}^1  (X(N), \bar{\mb{F}}_p) $$

\begin{coro} 
\label{comp}
We have an injection of representations
$$ (\widehat{H}^1 _{ss, c, \bar{\mb{F}}_p})_{(\pi)} \hookrightarrow (\widehat{H}^1 _{\bar{\mb{F}}_p})_{(\pi)}$$
\end{coro}

We  define also for a future use
$$\widehat{H}^1 _{ord, \bar{\mb{F}}_p} = \varinjlim _N \varinjlim _m H^1 (X(Np^m)_{ord}, \bar{\mb{F}}_p)$$
and for $a \in \mb{Z}_p ^{\times}$
$$\widehat{H}^1 _{a, \infty, \bar{\mb{F}}_p} = \varinjlim _N \varinjlim _m H^1 (\mb{X} _{a, \infty}(Np^m), \bar{\mb{F}}_p)$$

\subsection{Cohomology without support} We can apply similar reasoning as above to the situation without compact support. The roles of the ordinary locus and the supersingular locus are interchanged. By using again the decomposition of the ordinary locus and lemmas 3.3 and 3.4, we get that the localisation of $\widehat{H} ^1 _{X_{ord}}$ vanishes
$$\widehat{H} ^1 _{X_{ord}} (X(N), \bar{\mb{F}}_p) _{(\pi)} = 0$$
and hence we get
\begin{theo} We have an injection of representations
$$(\widehat{H}^1 _{\bar{\mb{F}}_p})_{(\pi)} \hookrightarrow \widehat{H}^1 _{ss, \bar{\mb{F}}_p}$$
where $\widehat{H}^1 _{ss, \bar{\mb{F}}_p}$ is defined similarly as above. 
\end{theo}
Later on, we will show that $\widehat{H}^1 _{ss, \bar{\mb{F}}_p}$ is a non-admissible representation, and this is why we cannot localise it a $\pi$. Let us finish by giving another definition for a future use (where $a \in \mb{Z}_p ^{\times}$)
$$\widehat{H}^1 _{\mb{X}_{a,\infty}, \bar{\mb{F}}_p} = \varinjlim _N \varinjlim _m H^1 _{\mb{X}_{a,\infty}(Np^m)}(X(Np^m)^{an}, \bar{\mb{F}}_p)$$

\section{New vectors}

Because there does not exist at the moment the Colmez functor in the context of quaternion algebras, which would be similar to the one considered for example in \cite{pa1}, we are forced to give a global definition of the mod p Jacquet-Langlands correspondence. To do that, we prove an analogue of a classical theorem of Casselman in the context of the modified mod l Langlands correspondence of Emerton-Helm (see \cite{eh}), which amounts to the statement that for any prime $l \not = p$, and for any local two-dimensional Galois representation $\rho$ of $\Gal(\bar{\mb{Q}}_l / \mb{Q}_l)$, there exists a compact, open subgroup $K_l \subset \GL_2(\mb{Z}_l)$ such that  $\pi _l(\rho) ^{K_l}$ has dimension 1, where $\pi _l(\rho)$ is the mod p representation of $\GL_2(\mb{Q}_l)$ associated to $\rho$ by \cite{eh}.

Let $\mf{b}$ be an ideal of $\mb{Z}_p$ and put $\Gamma _0 (\mf{b}) = \{ (\begin{smallmatrix} a & b \\ c & d \end{smallmatrix}) \in \GL_2(\mb{Z}_p) | c \equiv 0 \mod \mf{b} \}$. Let us recall the classical result of Casselman (see \cite{cas}):

\begin{theo}
Let $\pi$ be an irreducible admissible infinite-dimensional representation of $\GL_2(\mb{Q}_p)$ on $\bar{\mb{Q}}_l$-vector space and let $\epsilon$ be the central character of $\pi$. Let $c(\pi)$ be the conductor of $\pi$ which is the largest ideal of $\mb{Z}_p$ such that the space of vector $v$ with $\pi ( \begin{smallmatrix} a & b \\ c & d \end{smallmatrix}) v = \epsilon (a) v$, for all $ (\begin{smallmatrix} a & b \\ c & d \end{smallmatrix}) \in \Gamma _0 (c(\pi))$ is not empty. Then this space has dimension one. 
\end{theo}

We will prove that the result holds also modulo $p$ for the modified mod l Langlands correspondence. For that we need to assume that our prime $p$ is odd.

\begin{theo}
Let $\pi = \pi (\rho)$ be the mod $p$ admissible representation of $\GL_2(\mb{Q}_l)$ associated by the modified mod l Langlands correspondence to a Galois representation $\rho : G_{\mb{Q}_l} \ra \GL_2(\bar{\mb{F}} _p)$. Then there exists an open, compact subgroup $K$ of $\GL_2(\mb{Q}_l)$ such that $\dim _{\bar{\mb{F}}_p} \pi ^K =1$.
\end{theo}
\begin{proof}
We recall the results of \cite{eh} concerning the construction of the modified mod l Langlands correspondence. By Proposition 5.2.1 of \cite{eh}, the theorem is true when $\rho ^{ss}$ is not a twist of $1\oplus |\cdot |$, by the reduction modulo $p$ of the classical result of Casselman from \cite{cas}, which in the $l \not =p$ situation was proved by Vigneras in \cite{vi2} (see Theorem 23 and Proposition 24). When this is not the case, we can suppose that in fact $\rho ^{ss} = 1\oplus |\cdot |$ and we go by case-by-case analysis of the possible forms of $\pi(\rho)$ as described in \cite{eh} after Proposition 5.2.1 and in \cite{he}. The $\pi(\rho)$'s which appear are mostly extensions of four kinds of representations (and some combinations of them): trivial representation $1$, $|\cdot| \circ \det$, the Steinberg $\St$, $\pi(1)$ of Vigneras (see \cite{vi2}).

1) Suppose $0 \ra \pi(1) \ra \pi (\rho) \ra 1 \ra 0$. In this case $l \equiv -1 \mod p$. Let $\Gamma _0(p) =  \{ (\begin{smallmatrix} a & b \\ c & d \end{smallmatrix}) \in \GL_2(\mb{Z}_p) | c \equiv 0 \mod p, a \equiv d \equiv 1 \mod p  \} $. Then we have a long exact sequence associated with higher invariants by $\Gamma _0(p)$:
$$ 0 \ra \pi (\rho) ^{\Gamma _0(p)} \ra 1 \ra R^1 \pi(1) ^{\Gamma _0(p)}$$
as $\pi(1) ^{\Gamma _0(p)} = 0$ by the Proposition 24 of \cite{vi2}. We conclude by observing that $R^1 \pi(1) ^{\Gamma _0(p)} = 0$ because $|\Gamma _0(p)| = p^{\infty}$ and $l \not | |\Gamma _0(p)|$ by our assumption.

2) In the same way we deal with the situation when $\pi (\rho)$ is an extension of $|\cdot | \circ \det$ by $\pi(1)$ with the same assumption on $l$.

3) When $l \equiv -1 \mod p$ it is also possible to have $0 \ra \pi(1) \ra \pi (\rho) \ra 1 \oplus |\cdot | \circ \det  \ra 0$. Look at $\GL_2(\mb{Z}_p)$-invariants. The associated long exact sequence is
$$ 0 \ra \pi(\rho) ^{\GL_2(\mb{Z}_p)} \ra (1 \oplus |\cdot | \circ \det)^{\GL_2(\mb{Z}_p)} \ra R^1 \pi(1) ^{\GL_2(\mb{Z}_p)} = \Ext^1 _{\GL_2(\mb{F}_p)}(1, \pi(1))$$
Let us denote by $\mc{E}$ the extension of $1$ by $\pi(1)$ which we get from $0 \ra \pi(1) \ra \pi (\rho) \ra 1 \oplus |\cdot | \circ \det  \ra 0$. The last map in the above exact sequence is explicit
$$(1 \oplus |\cdot | \circ \det)^{\GL_2(\mb{Z}_p)} \ra \Ext^1 _{\GL_2(\mb{F}_p)}(1, \pi(1))$$
$$ (a,b) \mapsto (a+b)\mc{E}$$
and we see that it gives a line in $\Ext^1 _{\GL_2(\mb{F}_p)}(1, \pi(1))$ and hence the kernel, i.e. $\pi(\rho) ^{\GL_2(\mb{Z}_p)}$, is  one-dimensional as $(1 \oplus |\cdot | \circ \det)^{\GL_2(\mb{Z}_p)}$ has dimension two.

4) The last non-banal case with which we have to deal is the case when $p$ is odd, $l \equiv 1 \mod p$ and we have an extension:
$$ 0 \ra \St \ra \pi(\rho) \ra 1 \ra 0$$
In this case $\Ext ^1 _{\GL_2(\mb{Q}_p)}(1, \St)$ is two-dimensional (see Lemma 4.2 in \cite{he}). We look at the reduction map
$$\Ext ^1 _{\GL_2(\mb{Q}_p)}(1, \St) \ra \Ext ^1 _{\GL_2(\mb{F}_p)}(1, \St)$$
Let us denote by $\mc{E}$ the image of the class $[\pi(\rho)]$ of $\pi(\rho)$ in $\Ext ^1 _{\GL_2(\mb{F}_p)}(1, \St)$ under the above reduction. We have two cases to consider. Suppose firstly that $\mc{E} = 0$. Then we claim that $K=\GL _2(\mb{Z}_p)$ works. Indeed we have in this case
$$0 \ra \pi(\rho) ^K \ra 1 ^K \ra \Ext ^1 _{K}(1, \St) $$ 
and as the image of $1 ^K$ in $\Ext ^1 _{K}(1, \St) $ is $\mc{E}$, we conclude by assumption.

Now let us suppose that $\mc{E} \not =0$. Then we claim that the Iwahori subgroup $K = I$ works. We have
$$0 \ra \St ^K \ra \pi(\rho) ^K \ra 1^K \ra \Ext ^1 _{K}(1, \St)$$ 
The image of $1^K$ in $\Ext ^1 _{K}(1, \St) $ is non-zero by assumption, because $\Ext ^1 _{\GL_2(\mb{Z}_p)}(1, \St) \hookrightarrow \Ext ^1 _K (1, \St)$. Hence $\pi(\rho)^K$ is isomorphic to $\St ^K$ which is of dimension one. 

5) We remark that there is also the so-called banal case when $l$ is not congruent to $\pm 1$ modulo $p$. In this case, there are two situations to consider. In the first one $\pi (\rho) = \St \otimes | \cdot | \circ \det$ and we can take $K = I$, the Iwahori subgroup. In the second one $\pi (\rho)$ is the unique non-split extension of $| \cdot | \circ \det$ by $\St \otimes | \cdot | \circ \det$. Because $\Ext ^1 _{\GL_2(\mb{F}_p)}(1, \St) = 0$ as we are in the banal case, we conclude as above that $K = \GL_2 (\mb{Z}_p)$ works. 
\end{proof}

\section{The fundamental representation}

Following the original Deligne's approach to the non-abelian Lubin-Tate theory, we define the local fundamental representation. Using it, we refine the Lubin-Tate side of the injections we have considered. Then we recall Emerton's results on the cohomology of the tower of modular curves, yielding by a comparison an information on the local fundamental representation. Our arguments are similar to those given in \cite{de}. 

\subsection{Cohomology of the supersingular tube} We have introduced in section 2.3, the set $\Delta$, spaces $LT_{\Delta / K_m} = \coprod _{\delta \in \Delta / K_m} LT _{\delta}$ and we have obtained a description of the supersingular tube
$$X(Np ^m) _{ss} \simeq LT _{\Delta / K_m} \times _{D^{\times}(\mb{Q})} \GL_2(\mb{A} _f ^p) / K(N)$$

\begin{defi} Define the fundamental representation by
$$\widehat{H}^1 _{LT, c, \bar{\mb{F}}_p} = \varinjlim _m  H^1 _c(LT_{\Delta / K_m}, \bar{\mb{F}}_p)$$
Similarly we introduce the fundamental representation without support denoting it by $\widehat{H}^1 _{LT, \bar{\mb{F}}_p}$.
\end{defi}

From the description of supersingular points, we have
$$ H^1 _c(X(Np^m) _{ss}, \bar{\mb{F}}_p) = H^1_c(LT _{\Delta / K_m} \times _{D^{\times}(\mb{Q})} \GL_2(\mb{A} _f ^p)/K(N), \bar{\mb{F}}_p) =$$
$$= \{f: D^{\times}(\mb{Q}) \backslash \GL_2(\mb{A} _f ^p) / K(N) \ra H^1_c(LT _{\Delta / K_m}, \bar{\mb{F}}_p) \}$$
We take the direct limit:
$$ \varinjlim _{m} H^1 _c(X(Np^m) _{ss}, \bar{\mb{F}}_p) \simeq \{f: D^{\times}(\mb{Q}) \backslash \GL_2(\mb{A} _f ^p) / K(N) \ra \varinjlim _m  H^1_c(LT _{\Delta / K_m}, \bar{\mb{F}}_p) \}$$
Take the limit over $N$ to obtain 
$$\widehat{H}^1 _{ss, c, \bar{\mb{F}}_p} \simeq \{f: D^{\times}(\mb{Q}) \backslash \GL_2(\mb{A} _f ^p) \ra \widehat{H}^1_{LT, \bar{\mb{F}}_p} \} \simeq $$
$$ \simeq  \{f: D^{\times}(\mb{Q}) \backslash D^{\times}(\mb{A} _f) \ra \widehat{H}^1_{LT, c, \bar{\mb{F}}_p} \} ^{D^{\times}(\mb{Q} _p)} \simeq $$
$$ \simeq  \left( \{f: D^{\times}(\mb{Q}) \backslash D^{\times}(\mb{A} _f) \ra \bar{\mb{F}} _p \} \otimes _{\bar{\mb{F}}_p} \widehat{H}^1_{LT, c, \bar{\mb{F}}_p} \right) ^{D^{\times}(\mb{Q} _p)}$$

Let 
$$\mbf{F} =  \{f: D^{\times}(\mb{Q}) \backslash D^{\times}(\mb{A} _f) \ra \bar{\mb{F}} _p \}$$
where $f$ are locally constant functions, then
\begin{equation} \widehat{H}^1 _{ss, c, \bar{\mb{F}}_p} \simeq \left( \mbf{F} \otimes _{\bar{\mb{F}}_p} \widehat{H}^1_{LT, c, \bar{\mb{F}}_p} \right) ^{D^{\times}(\mb{Q} _p)} \label{eq: H1} \end{equation}
We get a similar result for the cohomology without support
$$\widehat{H}^1 _{ss,  \bar{\mb{F}}_p} \simeq \left( \mbf{F} \otimes _{\bar{\mb{F}}_p} \widehat{H}^1_{LT,  \bar{\mb{F}}_p} \right) ^{D^{\times}(\mb{Q} _p)}$$

\subsection{Emerton's results} We recall Emerton's results on the completed cohomology of modular curves. Remark that we are using implicitly the comparison theorem for \'{e}tale cohomology of a scheme and its analytification which is proved in \cite{ber}.

Let us fix a finite set $\Sigma = \Sigma _0 \cup \{p\}$. Let $K^{\Sigma} = \prod _{l \not \in \Sigma} K_l$ where $K_l = \GL_2(\mb{Z} _l)$ and choose an open, compact subgroup $K_{\Sigma _0}$ of $\prod _{l \in \Sigma _0} \GL_2(\mb{Z}_l)$. Let $\bar{\rho} : G _{\mb{Q}} \ra \GL_2(\bar{\mb{F}} _p)$ be an odd, irreducible, continuous representation unramified outside $\Sigma$. Remark that by Serre's conjecture (see \cite{kh}) $\bar{\rho}$ is modular. Let us denote by $\mf{m}$ the maximal ideal in the Hecke algebra $\mb{T}(K_{\Sigma _0})$ which corresponds to $\bar{\rho}$. We write also $\bar{\rho} _{| G_{\mb{Q}_p}} = \Ind _{G_{\mb{Q}_{p^2}}} ^{G_{\mb{Q}}} \alpha$, where $\alpha$ can be considered as a character of $\mb{Q} _{p^2} ^{\times}$ by the local class field theory. For the definitions, see Section 5 of \cite{em2}.   

\begin{theo} Assuming that $\bar{\rho}$ satisfies certain technical hypotheses (see the proof below), we have an isomorphism
$$\widehat{H}^1 _{\bar{\mb{F}} _p} [\mf{m}] ^{K^{\Sigma}} \simeq \pi \otimes _{\bar{\mb{F}}_p} \pi _{\Sigma _0} (\bar{\rho}) \otimes _{\bar{\mb{F}}_p} \bar{\rho}$$
where $\pi$ is a representation of $\GL_2(\mb{Q}_p)$ associated to $\bar{\rho}$ by the mod $p$ local Langlands correspondence and $\pi _{\Sigma _0}(\bar{\rho})$ is a representation $\GL_2(\mb{A} _f ^{\Sigma_0} )$ associated to $\bar{\rho}$ by the modified local Langlands correspondence mod $l$ for $l \in \Sigma _0$ (see \cite{eh}).
\end{theo}
\begin{proof}
For the exact assumptions, see Proposition 6.1.20 in \cite{em2}. Those assumptions are not important for our applications, as we can always find $\bar{\rho}$ which satisfies them and which at $p$ is isomorphic to our fixed irreducible Galois representation $\bar{\rho}_p$ (see below). 
\end{proof}

\subsection{Comparison} We will use results of Emerton to describe a part of $\widehat{H}^1 _{ss, \bar{\mb{F}}_p}$. We start by comparing mod p Hecke algebras for $\GL_2$ and for $D^{\times}$. On $\mbf{F}$, after taking $K^{\Sigma}$-invariants, there is an action of a Hecke algebra. For $l \not \in \Sigma$, we have a Hecke operator $T_l$ acting on functions of $D^{\times} (\mb{A} _f)$ by
$$T_l(f)(x) = f(xg) + \sum _{i=0} ^{l-1} f(xg_i)$$
where $g = (\begin{smallmatrix} l & 0 \\ 0 & 1 \end{smallmatrix})$ and $g_i = (\begin{smallmatrix} 1 & 0 \\ i & l \end{smallmatrix})$ are both considered as elements of $D^{\times}(\mb{A} _f)$ having $1$ at places different from $l$. Let us denote by $\mb{T}^D(K_{\Sigma _0})$ the (completed) Hecke algebra, which is the free $\mc{O}$-algebra spanned by the operators $T_l$ and  $S_l$ for all $l \not \in \Sigma$, where $S_l = [K_{\Sigma _0} K^{\Sigma} (\begin{smallmatrix} \varpi & 0 \\ 0 & \varpi \end{smallmatrix}) K_{\Sigma _0} K^{\Sigma}]$. By the results of Serre (see letter to Tate from \cite{se}), systems of eigenvalues for $(T_l)$ of $\mb{T} ^D(K_{\Sigma _0})$ on $\mbf{F}$ are in bijection with systems of eigenvalues for $(T_l)$ of $\mb{T}(K_{\Sigma _0})$ coming from mod p modular forms. This allows us to identify maximal ideals of $\mb{T}^D(K_{\Sigma _0})$ with those of $\mb{T}(K_{\Sigma _0})$ and in what follows we will make no distinction between them. 

Let $\bar{\rho} _p$ be the local Galois representation associated to a supersingular representation $\pi$ of $\GL_2(\mb{Q} _p)$ by the mod p Langlands correspondence. We assume that there exists a representation $\bar{\rho} : G _{\mb{Q}} \ra \GL_2(\bar{\mb{F}} _p)$ which is odd, irreducible, continuous, unramified outside a finite set $\Sigma = \Sigma _0 \cup \{p\}$, and such that $\bar{\rho} _{|G_{\mb{Q}_p}} = \bar{\rho} _p$. This is always the case as we may see from the description of the reductions of Galois representations associated to modular forms - see the introduction to \cite{br2} for a discussion (especially Conjecture 1.5) and compare it with the main result of \cite{bg}. 

Let us denote by $\mf{m}$ the maximal ideal in the Hecke algebra $\mb{T}(K _{\Sigma _0})$ corresponding to $\bar{\rho}$. Results of Emerton apply to $\bar{\rho}$ because we have assumed that $\bar{\rho} _p$ is irreducible. We denote by $K_{\mf{m}, \Sigma _0}$ an open compact subgroup of $\prod _{l \in \Sigma _0} \GL_2(\mb{Z} _l)$ for which $\pi _{\Sigma _0}(\bar{\rho})^{K_{\mf{m},\Sigma _0}}$ is a one-dimensional vector space (new vectors). We put $K _{\mf{m}} = K_{\mf{m}, \Sigma _0} K^{\Sigma}$ and we define:
$$ \sigma _{\mf{m}} = \mbf{F} [\mf{m}] ^{K_{\mf{m}}}$$
This is a representation of $D^{\times}(\mb{Q}_p)$. We remark that Breuil and Diamond in \cite{bd} also define a representation of $D^{\times}(\mb{Q}_p)$ which serves as a model for a local representation which should appear conjecturally at the place $p$ in the local-global compatibility of the Buzzard-Diamond-Jarvis conjecture (see the next section for a discussion). Their construction is different from our and uses "`types"' instead of new vectors. 

Let us look again at our cohomology groups. Taking $K_{\mf{m}}$-invariants, which commute with $D^{\times}(\mb{Q}_p)$-invariants, we get
$$\left( \widehat{H}^1 _{ss, \bar{\mb{F}}_p} \right)^{K_{\mf{m}}} \simeq \left( \mbf{F} ^{K_{\mf{m}}} \otimes _{\bar{\mb{F}}_p}  \widehat{H}^1_{LT, \bar{\mb{F}}_p} \right) ^{D^{\times}(\mb{Q} _p)}$$
Let us define the dual $\sigma _{\mf{m}} ^{\vee} = \Hom _{\bar{\mb{F}}_p}(\sigma _{\mf{m}}, \bar{\mb{F}}_p)$. It is not neccesarily a smooth representation. Taking $[\mf{m}]$-part we get:
$$\left( \widehat{H}^1 _{ss, \bar{\mb{F}} _p} [\mf{m}] \right) ^{K _{\mf{m}}} \simeq \left( \sigma _{\mf{m}} \otimes _{\bar{\mb{F}}_p}  \widehat{H}^1_{LT, \bar{\mb{F}}_p} \right) ^{D^{\times}(\mb{Q} _p)} =: \widehat{H}^1_{LT, \bar{\mb{F}}_p} [\sigma _{\mf{m}} ^{\vee}]$$
Thus, by the results proven earlier, we have
$$\pi \otimes _{\bar{\mb{F}}_p} \bar{\rho} \simeq \left( \widehat{H}^1 _{\bar{\mb{F}} _p} [\mf{m}]  \right) ^{K _{\mf{m}}} _{(\pi)} \hookrightarrow \left( \widehat{H}^1 _{ss,  \bar{\mb{F}} _p} [\mf{m}] \right) ^{K _{\mf{m}}}  \simeq \widehat{H}^1_{LT, \bar{\mb{F}}_p} [\sigma _{\mf{m}} ^{\vee}]$$
and we arrive at
\begin{theo}
\label{rough}
We have a $\GL_2(\mb{Q}_p) \times G_{\mb{Q}_p}$-equivariant injection:
$$\pi \otimes _{\bar{\mb{F}}_p} \bar{\rho} \hookrightarrow \widehat{H}^1_{LT, \bar{\mb{F}}_p} [\sigma _{\mf{m}} ^{\vee}] $$
\end{theo}
We will strengthen this result after proving additional facts about $\sigma _{\mf{m}}$. It is also possible to obtain the analogous result in the p-adic setting. Details will appear elsewhere.

\subsection{The mod p Jacquet-Langlands correspondence}
We have defined above
$$ \sigma _{\mf{m}} = \mbf{F} [\mf{m}] ^{K_{\mf{m}}}$$
This is a mod p representation of $D ^{\times}(\mb{Q}_p)$ which is one of our candidates for the mod p Jacquet-Langlands correspondence we search for. We will analyse this representation more carefully in the next section, getting a result about its socle. The question we do not answer here is whether this local representation is independent of the Hecke ideal $\mf{m}$ and if yes, how to construct it by local means. We make a natural conjecture
\begin{conj} Let $\mf{m}$ and $\mf{m}'$ be two maximal ideals of the Hecke algebra, which correspond to Galois representations $\bar{\rho}$ and $\bar{\rho}'$ such that $\bar{\rho} _p \simeq \bar{\rho} ' _p$. Then we have a $D^{\times}(\mb{Q}_p)$-equivariant isomorphism
$$ \sigma _{\mf{m}} \simeq \sigma _{\mf{m}'}$$
\end{conj}
This conjecture is natural in view of the fact that $\sigma _{\mf{m}}$ should play a role of the mod p Jacquet-Langlands correspondence and it should depend only on a local data. In fact, this conjecture follows from the local-global compatibility part of the Buzzard-Diamond-Jarvis conjecture (see Conjecture 4.7 in \cite{bdj})
\begin{conj} We have a $D^{\times}(\mb{A})$-equivariant isomorphism
$$\mbf{F}[\mf{m}] \simeq \sigma  \otimes \pi^p(\bar{\rho})$$
where $\sigma$ is a $D^{\times}(\mb{Q}_p)$-representation which depends only on $\bar{\rho}_p$, where $\bar{\rho}$ is the Galois representation associated to $\mf{m}$.
\end{conj}
The conjecture of Buzzard-Diamond-Jarvis would be proved if one could show the existence of an analogue of the Colmez functor in the context of quaternion algebras. Then, the methods of Emerton from \cite{em2} could be applied to give a proof. We come back to the discussion of the mod p Jacquet-Langlands correspondence at the end of the next section.

\section{Representations of quaternion algebras: mod p theory}

In this section we analyse more carefully mod p representations of quaternion algebras, especially representations $\sigma _{\mf{m}}$ defined in the preceding section. We also define a naive mod p Jacquet-Langlands correspondence.

\subsection{Naive mod p Jacquet-Langlands correspondence} By the work of Vigneras (see \cite{vi}), we know that all irreducible representations of $D^{\times}$ are of dimension 1 or 2 and are either 

1) a character of $D^{\times}(\mb{Q}_p)$, or

2) are of the form $\Ind ^{D^{\times}} _{\mc{O} _{D} ^{\times} \mb{Q} _{p^2} ^{\times}} \alpha$ where $\alpha$ is a character of $\mb{Q} _{p^2} ^{\times}$. 

Let $\bar{\rho} _p$ be the mod p 2-dimensional irreducible Galois representation which corresponds to the supersingular representation $\pi$ of $\GL_2(\mb{Q}_p)$ by the mod p Local Langlands correspondence. As we have mentioned earlier, it is of the form $\Ind _{G_{\mb{Q} _{p^2}}} ^{G_{\mb{Q}_p}} (\omega _2 ^{r} \cdot \sgn) \otimes \chi$ where $\chi$ is a character and $r \in \{1, ..., p\}$. 

\begin{defi}
The naive mod p Jacquet-Langlands correspondence is
$$\Ind _{G_{\mb{Q} _{p^2}}} ^{G_{\mb{Q}_p}} (\omega _2 ^{r} \cdot \sgn) \otimes \chi \mapsto \Ind ^{D^{\times}} _{\mc{O} _{D} ^{\times} \mb{Q} _{p^2}} (\omega _2 ^{r}) \otimes \chi$$
where $\omega _2 ^r$ is treated as a character of $\mb{Q}_{p^2}$ by the local class field theory and $\chi$ is considered both as a character of $G_{\mb{Q}_p}$ and $D^{\times}(\mb{Q}_p)$. This gives a bijection between two-dimensional representations of $G_ {\mb{Q}_p}$ and two-dimensional representations of $D^{\times}(\mb{Q}_p)$. Similar correspondence holds for characters.
\end{defi}
We remark that one may also would like to call this the naive mod p Langlands correspondence for $D^{\times}(\mb{Q}_p)$. We get the Jacquet-Langlands correspondence in the usual sense, when we compose it with the mod p local Langlands correspondence for $\GL_2(\mb{Q}_p)$. 

Let $\alpha : \mb{Q} _{p^2} \ra \bar{\mb{F}}_p ^{\times}$ be a character. We denote by $\rho (\alpha)$ the representation of $G _{\mb{Q}_p}$ obtained by the local class field theory and an induction. We denote by $\sigma (\alpha)$ the $D^{\times}(\mb{Q}_p)$-representation $\Ind ^{D^{\times}} _{\mc{O} _{D} ^{\times} \mb{Q} _{p^2}} (\alpha)$. We remark that we also could define the naive mod p Jacquet-Langlands correspondence as
$$\rho (\alpha) \mapsto \sigma (\alpha)$$
but we have chosen our normalisation with a twist by $\sgn$ to have the same condition on determinants as for the classical l-adic Jacquet-Langlands correspondence.

\subsection{Quaternionic forms} Let $D$ be the quaternion algebra over $\mb{Q}$, ramified at $p$ and at $\infty$. Let $K$ be a finite extension of $\mb{Q}_p$ with ring of integers $\mc{O}$ and a uniformiser $\varpi$. Define
$$\mbf{F} = \varinjlim _{K} H^0(D^{\times}(\mb{Q})\backslash D^{\times}(\mb{A} _f)/K, \bar{\mb{F}} _p)$$
$$\mbf{F}_{\mc{O}} = \varinjlim _{K} H^0(D^{\times}(\mb{Q})\backslash D^{\times}(\mb{A} _f)/K, \mc{O})$$
Define also $\mbf{F} _{K} = \mbf{F} _{\mc{O}} \otimes _{\mc{O}} K$. We can make similar definitions for other $\mb{F}_p$ or $\mb{Z}_p$-algebras (for example for finite extensions of $\mb{F}_p$ or for $\bar{\mb{Z}}_p$ in $\mbf{F} _{\bar{\mb{Z}}_p}$ which we will use in the text).
 
Recall that we have fixed a finite set $\Sigma = \Sigma _0 \cup \{p\}$ and chosen an open, compact subgroup $K_{\Sigma _0}$ of $\prod _{l \in \Sigma _0} \GL_2(\mb{Z}_l)$. On each of the above spaces, after taking $K^{\Sigma}$-invariants, there is an action of the Hecke algebra $\mb{T}^D(K_{\Sigma _0})$. Recall also that we have defined $\bar{\rho} : G _{\mb{Q}} \ra \GL_2(\bar{\mb{F}} _p)$ an odd, irreducible, continuous representation unramified outside $\Sigma$ and we have denoted by $\mf{m}$ the maximal ideal in $\mb{T}(K_{\Sigma _0})$ (or in $\mb{T}^D(K_{\Sigma _0})$) which corresponds to $\bar{\rho}$. We write 
$$\bar{\rho} _{| G_{\mb{Q}_p}} = \rho (\alpha)$$ 
where $\alpha$ can be considered as a character of $\mb{Q} _{p^2} ^{\times}$ by the local class field theory.  

\begin{prop} Take an open, compact subgroup $K_p$ of $D^{\times}(\mb{Q}_p)$ and choose $K_{\Sigma _0}$ to be an open, compact subgroup of $\prod _{l \in \Sigma _0} \GL_2(\mb{Z}_l)$ such that $K_pK_{\Sigma _0} K^{\Sigma}$ is neat. Then $\mbf{F} _{\mf{m}} ^{K_{\Sigma _0} K^{\Sigma}}$ is injective as a smooth representation of $K_p$.
\end{prop}
We do not define the notion of neatness for which we refer to section 0.6 in \cite{pi}. We only need this condition to ensure that $K_p$ acts freely as in the proof below. Any sufficiently small open compact subgroup is neat.
\begin{proof}
Let $M$ be any smooth finitely generated representation of $K_p$. We have 
$$\mbf{F} ^{K_{\Sigma _0} K^{\Sigma}}=\varinjlim _{K_p '} \mbf{F} ^{K_p ' K_{\Sigma _0} K^{\Sigma}}$$
where $K_p ' \subset K_p$ runs over sufficiently small, normal open subgroups of $\mc{O} _D ^{\times}$, so that $K_p '$ acts trivially on $M$. We can associate to $M$ a local system $\mc{M}$ on $D^{\times}(\mb{Q})\backslash D^{\times} (\mb{A} _f) / K_{\Sigma _0} K^{\Sigma}$. Because $K_p $ acts freely on $D^{\times}(\mb{Q})\backslash D^{\times} (\mb{A} _f) / K_{\Sigma _0} K^{\Sigma}$ by the assumption of neatness, we can descend this system to each $D^{\times}(\mb{Q})\backslash D^{\times} (\mb{A} _f) / K_p' K_{\Sigma _0} K^{\Sigma}$, where $K_p'$ is as above. Moreover on each $D^{\times}(\mb{Q}) \backslash D^{\times} (\mb{A} _f) / K_p 'K_{\Sigma _0} K^{\Sigma}$, $\mc{M}$ is a constant local system and hence:
$$\Hom _{K_p} (M, \mbf{F} ^{K_{\Sigma _0} K^{\Sigma}}) \simeq \varinjlim _{K_p'} \Hom _{K_p}(M, \mbf{F} ^{K_p'K_{\Sigma _0} K^{\Sigma}}) \simeq \varinjlim _{K_p'}(\mbf{F} ^{K_p'K_{\Sigma _0} K^{\Sigma}}(\mc{M} ^{\vee})) ^{K_p} \simeq \mbf{F} ^{K_p ^{\times} K_{\Sigma _0} K^{\Sigma}} (\mc{M} ^{\vee})$$
where $\mbf{F}(\mc{M} ^{\vee}) = H^0(D^{\times}(\mb{Q}) \backslash D^{\times} (\mb{A} _f), \mc{M} ^{\vee})$. Because $\mbf{F} ^{K_p K_{\Sigma _0} K^{\Sigma}} (\mc{M} ^{\vee})$ is an exact functor (there is no $H^1$), we get the result.
\end{proof}

We will now start to analyse socles of quaternionic forms $\mbf{F} _{\mf{m}} ^{K_{\Sigma _0} K^{\Sigma}}$. Let us start with the following lemma:
\begin{lemm} Let $\beta$ be a finite dimensional $\bar{\mb{F}}_p$-representation of $\mc{O} _D ^{\times}$. We have $\Hom _{\mc{O} _{D} ^{\times}} (\beta ^{\vee},  \mbf{F} ^{K^p} _{\mf{m}}) \simeq \mbf{F} ^{K^p} _{\mf{m}} \{\beta \}$, where $\mbf{F} ^{K^p} _{\mf{m}} \{\beta \}$ is the space of automorphic functions $D(\mb{Q}) \backslash D(\mb{A}_f) / K^p \ra \beta$.
\end{lemm}
\begin{proof} The isomorphism is given by an explicit map. See Lemma 7.4.3 in \cite{egh}.
\end{proof}

\begin{prop} The only irreducible $\bar{\mb{F}}_p$-representations of $D^{\times}(\mb{Q}_p)$ which appear as submodules in $\mbf{F} _{\mf{m}} ^{K_{\Sigma _0} K^{\Sigma}}$ are isomorphic to $\sigma ^{\vee} = \sigma (\alpha) ^{\vee}$.
\end{prop}
\begin{proof}
Observe that the only irreducible $\bar{\mb{F}}_p$-representations of $\mc{O}_D ^{\times}$ which can appear in the $\mc{O} _D ^{\times}$-socle of $\mbf{F} ^{K^p} _{\mf{m}}$ are duals of the Serre weights of $\bar{\rho}$. This follows from the lemma above and the definition of being modular, i.e. $\bar{\rho}$ is modular of weight $\beta$ (where $\beta$ is a representation of $\mc{O} _{D} ^{\times}$) if and only if there exists an open compact subset $U$ of $D^{\times}(\mb{A}_f)$ such that $\mbf{F} ^{U} _{\mf{m}} \{\beta \} \not = 0$. By the lemma, this is equivalent to $\Hom _{\mc{O} _{D} ^{\times}} (\beta ^{\vee},  \mbf{F} ^{U} _{\mf{m}}) \not = 0$ which holds if and only if $\beta ^{\vee} \in soc _{\mc{O} _D ^{\times}} \mbf{F} ^{U} _{\mf{m}}$. Now the result follows from Theorem 7 in \cite{kh2}, as the only possible weights which can appear in the socle are $\alpha ^{\vee}$ and $(\alpha ^p)^{\vee}$. Hence the $D^{\times}(\mb{Q}_p)$-socle contains only $\sigma(\alpha) ^{\vee}$. 
\end{proof}
As a corollary we also get the $[\mf{m}]$-isotypic analogue of the above
\begin{coro}
The only irreducible representations which appear as submodules in $\mbf{F} ^{K_{\Sigma _0} K^{\Sigma}} [\mf{m}]$ are isomorphic to $\sigma ^{\vee} = \sigma (\alpha) ^{\vee}$.
\end{coro}
We are now ready to strengthen the theorem which has appeared before
\begin{theo}
The representation $\sigma \otimes \pi \otimes \bar{\rho}$ appears as a subquotient in $\widehat{H}^1 _{LT, \bar{\mb{F}}_p}$
\end{theo}
\begin{proof}
This follows from 
$$\pi \otimes \bar{\rho} \hookrightarrow \widehat{H}^1 _{LT, \bar{\mb{F}}_p}[\sigma _{\mf{m}} ^{\vee}]$$
and the fact that the only irreducible $D^{\times}(\mb{Q}_p)$-representation which appears as a quotient of $\sigma _{\mf{m}} ^{\vee}$ is $\sigma$.
\end{proof} 

We remark that if
$$n=\dim _{\bar{\mb{F}} _p} \Hom _{D^{\times}(\mb{Q}_p)} (\sigma (\alpha) ^{\vee}, \mbf{F} ^{K_{\Sigma _0} K^{\Sigma}} [\mf{m}])$$
then one conjectures that $n=1$ (even in the more general setting, see Section 8 of \cite{br3}). 

Before moving further, let us recall a structure theorem of Breuil and Diamond for our $D^{\times}(\mb{Q}_p)$-representations, which shows that our candidate for the mod p Jacquet-Langlands correspondence defined above is of entirely different nature than the one with complex coefficients. 
\begin{prop}
The $D^{\times}(\mb{Q}_p)$-representation $\mbf{F} ^{K_{\Sigma _0} K^{\Sigma}} [\mf{m}]$ is of infinite length.
\end{prop}
\begin{proof} We give a sketch of the proof, which is contained in \cite{bd} as Corollary 3.2.4 (it is conditional on the local-global compatibility part of the Buzzard-Diamond-Jarvis conjecture). Firstly observe that it is enough to prove that $\mbf{F} ^{K_{\Sigma _0} K^{\Sigma}} [\mf{m}]$ is of infinite dimension over $\bar{\mb{F}}_p$, because a representation of finite length will be also of finite dimension as $D^{\times}$ is compact modulo center. Suppose now that we have an automorphic form $\pi$ such that the reduction of its associated Galois representation $\bar{\rho} _{\pi}$ is isomorphic to $\bar{\rho}$ and $\pi ^{K_{\Sigma _0} K^{\Sigma}} \not = 0$. Then there is a lattice $\Lambda _{\pi} = \mbf{F} _{\bar{\mb{Z}}_p} ^{K_{\Sigma _0} K^{\Sigma}} \cap  \pi ^{K_{\Sigma _0} K^{\Sigma}}$ inside $\pi ^{K_{\Sigma _0} K^{\Sigma}}$. Its reduction $\bar{\Lambda} _{\pi} = \Lambda _{\pi} \otimes _{\bar{\mb{Z}}_p} \bar{\mb{F}}_p$ lies in $\mbf{F} ^{K_{\Sigma _0} K^{\Sigma}} [\mf{m}]$ so it is enough to prove that we can find automorphic representations $\pi$ as above with $\pi ^{K_{\Sigma _0} K^{\Sigma}}$ of arbitrarily high dimension. This is done by explicit computations of possible lifts in \cite{bd}. 
\end{proof}
This proposition indicates that $\widehat{H}^1_{LT, \bar{\mb{F}}_p}$ is a non-admissible smooth representation. 

\subsection{Non-admissibility} We have
\begin{prop} The $\GL_2(\mb{Q}_p)$-representations $\widehat{H}^1 _{ss, \bar{\mb{F}}_p}$ and $\widehat{H}^2 _{X_{ord}, \bar{\mb{F}}_p}$ are non-admissible smooth $\bar{\mb{F}}_p$-representations.
\end{prop}
\begin{proof}
If one of them would be admissible, then also the second would because of the exact sequence
$$\widehat{H}^1 _{X_{ord}, \bar{\mb{F}}_p} \ra \widehat{H}^1 _{\bar{\mb{F}}_p} \ra \widehat{H}^1 _{ss, \bar{\mb{F}}_p} \ra \widehat{H}^2 _{X_{ord}, \bar{\mb{F}}_p} \ra \widehat{H}^2 _{\bar{\mb{F}}_p}$$
It is enough to prove that $\widehat{H}^1 _{ss, \bar{\mb{F}}_p}$ is non-admissible, or even that $\widehat{H}^1 _{LT, \bar{\mb{F}}_p}$ is non-admissible. Let us look at the Hochschild-Serre spectral sequence for the Iwahori level $I$ of the Lubin-Tate tower
$$H^i(I,\widehat{H}^j _{LT, \bar{\mb{F}}_p}) \Rightarrow H^{i+j}_{LT, I, \bar{\mb{F}}_p}$$
where we have denoted by $H^{i+j}_{LT, I, \bar{\mb{F}}_p}$ the fundamental representation at $I$-level. Now observe that if $\widehat{H}^1 _{LT, \bar{\mb{F}}_p}$ were admissible, then $H^0(I,\widehat{H}^1 _{LT, \bar{\mb{F}}_p})$ would be of finite dimension. Because $H^1(I,\widehat{H}^0 _{LT, \bar{\mb{F}}_p})$ is of finite dimension (as $\widehat{H}^0 _{LT, \bar{\mb{F}}_p}$ is), this would mean that $H^{1}_{LT, I, \bar{\mb{F}}_p}$ is finite-dimensional. But geometrically Lubin-Tate tower at level $I$ is an annulus (this is a standard fact, one can prove it by methods of section 8.1) and hence $H^{1}_{LT, I, \bar{\mb{F}}_p}$ has to be of infinite dimension (see remark 6.4.2 in \cite{ber4}). This contradiction finishes the proof.
\end{proof}

\begin{coro} The $\GL_2(\mb{Q}_p)$-representation $\widehat{H}^1_{LT, \bar{\mb{F}}_p}$ is a non-admissible smooth $\bar{\mb{F}}_p$-representation.
\end{coro}
\begin{proof} Follows from the proposition above.
\end{proof}

\subsection{On mod p Jacquet-Langlands correspondence} We come once again to the discussion of the mod p Jacquet-Langlands correspondence. Remark that there are three possible candidates for the correspondence which appear in our work:
\newline
\newline
1) The 2-dimensional irreducible representation $\sigma$ of $D^{\times}(\mb{Q}_p)$ defined by the naive mod p Jacquet-Langlands correspondence.
\newline
\newline
2) The representation $\sigma _{\mf{m}}$ defined by global means and depending a priori on a maximal Hecke ideal $\mf{m}$. It is of infinite length as a representation of $D^{\times}(\mb{Q}_p)$ and contains $\sigma ^{\vee}$ in its socle.
\newline
\newline
3) The representation defined via the cohomology
$$\sigma _{LT} = \Hom _{G _{\mb{Q}_p} \times \GL_2(\mb{Q}_p)} (\bar{\rho} \otimes _{\bar{\mb{F}}_p} \pi, \widehat{H} ^1 _{LT, \bar{\mb{F}}_p})$$
By the results above, it contains $\sigma$ as a subquotient.
\newline
\newline
In the l-adic setting, we can define representations of $D^{\times}(\mb{Q}_p)$ in the similar way and it is known that $\sigma _{LT} \simeq \sigma _{\mf{m}} ^{\vee}$. Moreover $\sigma _{\mf{m}}$ in the l-adic setting is 2-dimensional (at least in the moderately ramified case). This is not the case in the mod p setting as we have showed that representations of 1) and 2) are different (one is 2-dimensional, the other is infinite-dimensional). The natural definition of the mod p correspondence seems to be $\sigma _{LT}$ and it is also natural to ask whether one has $\sigma _{LT} \simeq \sigma _{\mf{m}} ^{\vee}$ for each appropiate $\mf{m}$ as considered before.

\section{Cohomology with compact support}

In this section we will discuss what happens when we consider the cohomology with compact support. Our basic result is negative and it states that the first cohomology group with compact support of the fundamental representation $\widehat{H}^1 _{LT,c, \bar{\mb{F}}_p}$ does not contain any supersingular representation of $\GL _2(\mb{Q}_p)$ as a subrepresentation. This suprising result, which is very different from the situation known in the l-adic setting where $l\not =p$, leads to a similar exact sequence as we have considered for cohomology without support, but this time, we get that $\pi \otimes \bar{\rho}$ is contained in the $H^1$ of the ordinary locus. 

\subsection{Geometry at pro-p Iwahori level}
Let $K(1) = \left( \begin{smallmatrix}  1 + p \mb{Z}_p & p \mb{Z}_p \\ p \mb{Z}_p & 1+p\mb{Z}_p \end{smallmatrix} \right)$ and let $I(1) = \left( \begin{smallmatrix}  1 + p \mb{Z}_p &  \mb{Z}_p \\ p \mb{Z}_p & 1+p\mb{Z}_p \end{smallmatrix} \right)$ be the pro-p Iwahori subgroup. We let
$$ \mc{M} _{LT, K(1)} = \Spf R _{K(1)}$$
$$ \mc{M} _{LT, I(1)} = \Spf R _{I(1)}$$
be the formal models for the Lubin-Tate space at levels $K(1)$ and $I(1)$ respectively. We will compute $R _{I(1)}$ explicitely. This is also done in a more general setting in the work of Haines-Rapoport (see Corollary 3.4.3 in \cite{hr}) but here we give a short and elementary argument. 

We know that $R _{I(1)} = R_{K(1)} ^{I(1)}$ and hence we can use the explicit description of $R _{K(1)}$ by Yoshida to get the result (see Proposition 3.5 in \cite{yo}). Let $W = W(\bar{\mb{F}}_p)$ be the Witt vectors of $\bar{\mb{F}}_p$. There is a surjection $W[[\tilde{X}_1,\tilde{X}_2]] \twoheadrightarrow R _{K(1)}$ which maps $\tilde {X} _i$ to $X_i$ where $X_i$ ($i=1,2$) are local parameters for $R_{K(1)}$ which form a $\mb{F}_p$-basis of $\mf{m}_{R_{K(1)}}[p] = \{ x \in \mf{m} _{R_{K(1)}} | [p](x) =0 \}$, where $[p]$ is explained below. We will find parameters for $R _{I(1)} = R_{K(1)} ^{I(1)}$. Observe that for $b\in \mb{F}_p$ we have (see chapter 3 of \cite{yo})
$$ \begin{pmatrix} 1  & b \\ 0 & 1 \end{pmatrix} X_1 = X_1$$
$$  \begin{pmatrix} 1  & b \\ 0 & 1  \end{pmatrix} X_2 = [b]X_1 + _{\Sigma} X_2$$
where $+ _{\Sigma}$ is the addition on the universal deformation of the unique formal group over $\bar{\mb{F}}_p$ of height 2 and $[.]$ gives the structure of multiplication by elements of $W$ on the same universal deformation $\Sigma$. See Chapter 3 of \cite{yo} for details. We see that $X_2$ is not invariant under $I(1)$ and hence we define $X_2 ' = \prod _{b \in \mb{F}_p} ([b]X_1 + _{\Sigma} X_2)$ which is. We claim that $(X_1, X_2')$ are local parameters for $R_{I(1)}$. Indeed if $z$ belongs to $R_{I(1)} = R_{K(1)} ^{I(1)}$ then we may write it as $z = P(X_1) + X_2 Q(X_1, X_2)$, where $P \in W[[X_1]]$ and $Q \in W[[X_1,X_2]]$. As $P(X_1)$ is invariant under $I(1)$, we see that also $X_2 Q(X_1, X_2)$ has to be invariant under $I(1)$. Because of the action of $\left( \begin{smallmatrix} 1  & b \\ 0 & 1  \end{smallmatrix} \right)$ on $X_2$ described above and the fact that $R_{K(1)}$ is a regular local ring hence factorial, we see that $X _2 '$ divides $X_2 Q(X_1,X_2)$ (we use here the fact that $[b]X_1 + _{\Sigma} X_2$ and $[b']X_1 + _{\Sigma} X_2$ are not associated for $b \not = b'$; this follows from Proposition 4.2 in \cite{str}). This leads to $z = P(X_1) + X_2 ' Q'(X_1,X_2)$ for some $Q'$ which is $I(1)$-invariant and hence we conclude by induction and the fact that polynomials are dense in formal series (look at $z$ modulo powers of the maximal ideal $\mf{m} _{R_{I(1)}}$).

Let us observe that for $a \in \mb{F}_p ^{\times}$ we have for $i=1,2$: $[a]X_i = uX_i$, where $u$ is a unit in $R_{K(1)}$. Let us now look at the relation defining $R _{K(1)}$ inside $W[[\tilde{X}_1,\tilde{X}_2]]$ which appears in Proposition 3.5 of \cite{yo}. We have
$$ p = u \prod _{(a_1,a_2) \in \mb{F}_p ^2 \backslash \{0,0\}} ([a_1]X_1 +_{\Sigma} [a_2]X_2)$$
where $u$ is some unit in $R_{K(1)}$. Let us write $a \sim b$ whenever $a = ub$ for some unit $u$ in $R_{K(1)}$. Thus we have
$$p \sim \prod _{(a_1,a_2) \in \mb{F}_p ^2 \backslash \{0,0\}} ([a_1]X_1 +_{\Sigma} [a_2]X_2) \sim \left( \prod _{a_1 \in \mb{F}_p ^{\times}} [a_1]X_1 \right) \left( \prod _{a_1 \in \mb{F} _p} \prod _{a_2 \in \mb{F}_p ^{\times} } [a_2]([a_1/a_2]X_1 +_{\Sigma} X_2) \right) \sim $$
$$\sim  \left( \prod _{a_1 \in \mb{F}_p ^{\times}} [a_1]X_1 \right) \left( \prod _{a_2 \in \mb{F}_p ^{\times}} X_2 ' \right) \sim (X_1 X_2 ')^{p-1}$$
Hence we have $p = u' (X_1 X_2 ')^{p-1}$ for some unit $u'$ in $R_{K(1)}$ a priori, but we can see that $u'$ is in fact a unit in $R_{I(1)}$.  Because $W[[X,Y]]$ is a complete local ring with an algebraically closed residue field there exists a $(p-1)$-th root of $u'$, and hence we can write $p = (X_1' X_2 '')^{p-1}$. We want to conclude that this is the only relation in $R_{I(1)}$ which means that there exists a surjection 
$$ B = W[[\tilde{X}_1 ', \tilde{X} '' _2]] \twoheadrightarrow R_{I(1)}$$
with kernel $f = (\tilde{X}_1' \tilde{X}'' _2)^{p-1} - p$. First of all, observe that $R _{I(1)}$ and $B / f B$ are regular local rings of dimension 2 with a surjection $B/fB \twoheadrightarrow R_{I(1)}$. We claim that this map has to be neccessarily an injection also. Indeed, this holds for any surjective morphism $A \twoheadrightarrow R$ of regular local rings of the same dimension by using the fact that that for a regular local ring we have $\gr ^{\bullet} _{\mf{m} _A} A \simeq \Sym \mf{m} _A / \mf{m} ^2 _A$. This yields an isomorphism at the graded level which lifts to the level of rings. All in all, we conclude that  
\begin{prop} We have
$$R _{I(1)} \simeq W[[X,Y]] / ( (XY)^{p-1} - p)$$
\end{prop}
This means that $\mc{M} _{LT, I(1)}$ is made of $p-1$ copies of an open annulus in $\mb{P} ^1$ after a base change to $W[\sqrt[p-1]{p}]$:
$$R _{I(1)} \otimes _W  W[\sqrt[p-1]{p}] \simeq \prod _{i =1} ^{p-1} W[[X,Y]] / (XY - \sqrt[p-1]{p} \cdot \zeta _{p-1} ^i)$$
\subsection{Cohomology at pro-p Iwahori level}

We compute $H^1_c (\mc{M} _{LT, I(1)}, \bar{\mb{F}}_p)$ (we will omit $\bar{\mb{F}}_p$ from the notation in what follows). Let $\mc{A}$ be an open annulus in $\mb{P}^1$. We can write a long exact sequence
$$0 \ra H^0 _c (\mc{A}) \ra H^0  (\mb{P}^1) \ra H^0(\mb{P} ^1 \backslash \mc{A}) \ra H^1 _c(\mc{A}) \ra H^1(\mb{P}^1)$$
We know that
$$ H^1(\mb{P}^1) = H^0 _c(\mc{A})= 0$$
$$ \dim _{\bar{\mb{F}}_p} H^0(\mb{P}^1) = 1$$
$$ \dim _{\bar{\mb{F}}_p} H^0 (\mb{P} ^1 \backslash \mc{A}) = 2$$
and hence it follows that 
$$\dim _{\bar{\mb{F}}_p} H^1 _c(\mc{A})=1$$
Because geometrically $\mc{M} _{LT, I(1)}$ is made of $p-1$ copies of $\mc{A}$, we have
$$\dim _{\bar{\mb{F}}_p} H^1 _c(\mc{M} _{LT, I(1)})=p-1$$
Let $\mc{H} = \mc{H} _{\GL_2}(I(1)) = \bar{\mb{F}}_p[I(1)\backslash \GL_2(\mb{Q}_p) / I(1)]$ be the mod p Hecke algebra at the pro-p Iwahori level. Let $I$ be the Iwahori subgroup of $\GL_2(\mb{Z}_p)$. We look at the action of $\bar{\mb{F}}_p[I /I(1)] \simeq \bar{\mb{F}}_p[(\mb{F}_p ^{\times})^2]$ on the cohomology. We know by \cite{str} that it acts by determinant on connected components of $\mc{M} _{LT, K(1)}$ and hence on connected components of $\mc{M} _{LT, I(1)}$ so we have a decomposition of $H^1 _c(\mc{M} _{LT, I(1)})$ into $p-1$ pieces of dimension 1:
$$ H^1 _c(\mc{M} _{LT, I(1)}) = \bigoplus _{\chi : \mb{F} _p ^{\times} \ra \bar{\mb{F}}_p ^{\times}} H^1 _c(\mc{M} _{LT, I(1)}) _{\chi}$$
where $H^1 _c(\mc{M} _{LT, I(1)}) _{\chi}$ is the part of $H^1 _c(\mc{M} _{LT, I(1)})$ on which $\bar{\mb{F}}_p[(\mb{F}_p ^{\times})^2]$ acts through $\chi \circ \det$. 

\subsection{Vanishing result}

We will now prove that the supersingular representation $\pi$ does not appear in $\widehat{H}^1 _{LT, c, \bar{\mb{F}}_p}$. First of all, remark that it is enough to prove that the $\mc{H}$-module $\pi ^{I(1)}$ does not appear in $(\widehat{H}^1 _{LT, c, \bar{\mb{F}}_p})^{I(1)}$, because the functor $\pi \mapsto \pi ^{I(1)}$ induces a bijection between supersingular representations and supersingular Hecke modules (see \cite{vi3}). We have the Hochschild-Serre spectral sequence (see Appendix A)
$$ H^i(I(1), \widehat{H}^j _{LT,c, \bar{\mb{F}}_p}) \Rightarrow H^{i+j} _{LT, c, I(1), \bar{\mb{F}}_p}$$
where we have denoted by $H^{i+j} _{LT, c, I(1), \bar{\mb{F}}_p}$ the fundamental representation at $I(1)$-level. This gives a long exact sequence
$$ 0 \ra H^1 (I(1), \widehat{H}^0 _{LT,c, \bar{\mb{F}}_p}) \ra H^1 _{LT, c, I(1), \bar{\mb{F}}_p} \ra (\widehat{H}^1 _{LT, c, \bar{\mb{F}}_p})^{I(1)} \ra H^2 (I(1), \widehat{H}^0 _{LT,c, \bar{\mb{F}}_p})$$
Because $\widehat{H}^0 _{LT, c, \bar{\mb{F}}_p} =0$ as $H^0 _c(\mc{M} _{LT},\bar{\mb{F}}_p) = 0$ we have an $\mc{H}$-equivariant isomorphism
$$H^1 _{LT, c, I(1), \bar{\mb{F}}_p} \simeq (\widehat{H}^1 _{LT, c, \bar{\mb{F}}_p})^{I(1)}$$
This means that if $\pi ^{I(1)}$ appears in $(\widehat{H}^1 _{LT, c, \bar{\mb{F}}_p})^{I(1)}$ then it appears also in $H^1 _{LT, c, I(1), \bar{\mb{F}}_p}$. But because $H^1 _{LT, c, I(1), \bar{\mb{F}}_p}$ consists of multiple copies of $H^1 _c(\mc{M}_{LT, I(1)}, \bar{\mb{F}}_p)$, it is enough to show that $\pi ^{I(1)}$ does not appear in $H^1 _c(\mc{M}_{LT, I(1)}, \bar{\mb{F}}_p)$. To prove it, it suffices to show that no supersingular $\mc{H}$-module appear in  $H^1 _c(\mc{M}_{LT, I(1)}, \bar{\mb{F}}_p)$. Let $M$ be any supersingular $\mc{H}$-module. Then we know that it is 2-dimensional and of the form $M_2(0,z,\omega)$ as in Section 3.2 of \cite{vi3}, where $\omega$ is a character of $I/ I(1)$. If we write $I/I(1) = \mb{F}_p^{\times} \times \mb{F}_p ^{\times}$ and $\omega = \eta _1 \otimes \eta _2$ then $M = (\eta _1 \otimes \eta _2) \oplus (\eta_2 \otimes \eta _1)$ as a $I/I(1)$-module. If $M$ appears in $H^1 _c(\mc{M}_{LT, I(1)}, \bar{\mb{F}}_p)$, then $I/I(1)$ acts on $M$ by determinant and hence $\eta _1 = \eta _2$. This would mean that $H^1 _c(\mc{M} _{LT, I(1)}) _{\eta _1}$ is at least 2-dimensional, which is a contradiction. All in all, we conclude that $\pi ^{I(1)}$ does not appear in $(\widehat{H} ^1 _{LT, c, \bar{\mb{F}}_p})^{I(1)}$ and hence
\begin{theo} The supersingular representation $\pi$ does not appear in $\widehat{H}^1 _{LT,c, \bar{\mb{F}}_p}$. 
\end{theo} 
We could rephrase it also as 
$$\widehat{H}^1 _{LT, c, \bar{\mb{F}}_p, (\pi)} = 0$$
\begin{rema}
Observe that the above proof does not use in any particular form the fact that we are working with $\GL _2(\mb{Q}_p)$, besides the fact that the functor $\pi \mapsto \pi ^{I(1)}$ induces a bijection between supersingular representations and supersingular $\mc{H}$-modules. Apart from that, the results of Vigneras and Yoshida holds for $\GL _2(F)$ as well, where $F$ is a finite extension $\mb{Q}_p$ and show that there are no supersingular modules in the cohomology with compact support of the Lubin-Tate tower at the pro-p Iwahori level. This leads to the conclusion that supersingular representations of $\GL _2 (F)$ attached to these supersingular modules by the construction of Paskunas (see \cite{pa5}) do not appear in the cohomology with compact support of the Lubin-Tate tower at infinite level. We remark that, contrary to $F = \mb{Q}_p$ case, those supersingular representations constructed by Paskunas do not conjecturally give all the supersingular representations of $\GL _2 (F)$.
\end{rema}
The above theorem gives us, when combined with the exact sequence for the supersingular locus, an appearance of the mod p local Langlands correspondence in the cohomology of the ordinary locus (in contrast with the mod l situation).
\begin{coro} We have an $\GL_2(\mb{Q}_p)\times G_{\mb{Q}_p}$-equivariant injection
$$ \pi \otimes \bar{\rho} \hookrightarrow \widehat{H}^1 _{ord, \bar{\mb{F}}_p}$$
\end{coro}
Moreover, this vanishing result can be used in the study of non-admissibility and in the description of the cohomology of certain Shimura curves.

\subsection{Non-admissibility} We will now show that our cohomology groups are non-admissible representations of $\GL_2(\mb{Q}_p)$. We start with:
\begin{prop} The $\GL_2(\mb{Q}_p)$-representations $\widehat{H}^2 _{ss, c, \bar{\mb{F}}_p}$ and $\widehat{H}^1 _{ord, \bar{\mb{F}}_p}$ are non-admissible smooth $\bar{\mb{F}}_p$-representations.
\end{prop}
\begin{proof}
If one of them would be admissible, then also the second would because of the exact sequence
$$\widehat{H}^1 _{ss, c, \bar{\mb{F}}_p} \ra \widehat{H}^1 _{\bar{\mb{F}}_p} \ra \widehat{H}^1 _{ord, \bar{\mb{F}}_p} \ra \widehat{H}^2 _{ss, c, \bar{\mb{F}}_p} \ra \widehat{H}^2 _{\bar{\mb{F}}_p}$$
But we know that $\widehat{H}^1 _{ord, \bar{\mb{F}}_p}$ is an induced representation
$$\Ind ^{\GL_2} _{B(\infty)} \left( \bigoplus _{a} \widehat{H}^1 _{a, \infty, \bar{\mb{F}}_p} \right)$$
so if it were admissible, then the localisation of it at $\pi$ would have to vanish. This is not possible by the corollary above.
\end{proof}
\begin{coro} The $\GL_2(\mb{Q}_p)$-representation $\widehat{H}^2_{LT, c, \bar{\mb{F}}_p}$ is a non-admissible smooth $\bar{\mb{F}}_p$-representation.
\end{coro}
\begin{proof} Follows from the proposition above.
\end{proof}

\subsection{Cohomology of Shimura curves} We will briefly sketch another consequence of vanishing of $\widehat{H}^1 _{LT, c, \bar{\mb{F}}_p, (\pi)}$. First of all, it implies that $H^1 _c (\mc{M} _{LT}, \bar{\mb{F}}_p) _{(\pi)}$ vanishes because $\widehat{H}^1 _{LT, c, \bar{\mb{F}}_p, (\pi)}$ is just a sum of copies of $H^1 _c (\mc{M} _{LT}) _{(\pi)}$. Now recall the Faltings isomorphism (see \cite{fa2}) which gives us 
$$H^1 _c (\mc{M} _{LT}, \bar{\mb{F}}_p) _{(\pi)} = H^1 _c (\mc{M} _{Dr}, \bar{\mb{F}}_p) _{(\pi)} = 0$$
where $\mc{M} _{Dr}$ is the Drinfeld tower at infinity (see \cite{da} for details). We have a spectral sequence coming from the p-adic uniformisation of the Shimura curve $Sh$ associated to the algebraic group $G''$ arising from the quaternion algebra over $\mb{Q}$ which is ramified precisely at $p$ and some other prime $q$:
$$E_2 ^{p,q} = \Ext ^p _{\GL_2(\mb{Q}_p)}(H^{2-q} _c(\mc{M} _{Dr , K_p}, \bar{\mb{F}}_p), C^{\infty}(G'(\mb{Q}) \backslash G'(\mb{A}), \bar{\mb{F}}_p)^{K^p}) \Rightarrow H^{p+q}_c(Sh ^{an} _{K_pK^p}, \bar{\mb{F}}_p)$$
where we have denoted by $G'$ the algebraic group arising from the quaternion algebra over $\mb{Q}$ which is ramified precisely at $q$ and $\infty$. For this, see \cite{fa} where it is proven for $\bar{\mb{Q}}_l$ but the proof works also for $\bar{\mb{F}}_p$ (the proof is also contained in the appendix B of \cite{da2}) .
 
Choose any non-Eisenstein maximal ideal $\mf{n}$ in the Hecke algebra of $G''$ whose associated Galois representation corresponds at $p$ to the supersingular representation $\pi$ we have chosen before. Take the direct limit over $K_p$ and localise the above spectral sequence at $\mf{n}$ to get
$$\Ext ^p _{\GL_2(\mb{Q}_p)}(H^{2-q} _c(\mc{M} _{Dr}, \bar{\mb{F}}_p) _{(\pi)}, C^{\infty}(G'(\mb{Q}) \backslash G'(\mb{A}), \bar{\mb{F}}_p)^{K^p} _{\mf{n}}) \Rightarrow H^{p+q}_c(Sh ^{an} _{K^p}, \bar{\mb{F}}_p) _{\mf{n}}$$
The localisation of $H^{2-q} _c(\mc{M} _{Dr}, \bar{\mb{F}}_p)$ at $\pi$ appears because $C^{\infty}(G'(\mb{Q}) \backslash G'(\mb{A}), \bar{\mb{F}}_p)^{K^p} _{\mf{n}}$ is $\pi$-isotypic. Using our vanishing result we get an interesting isomorphism
$$\Ext ^1 _{\GL_2(\mb{Q}_p)}(H^{2} _c(\mc{M} _{Dr}, \bar{\mb{F}}_p) _{(\pi)}, C^{\infty}(G'(\mb{Q}) \backslash G'(\mb{A}), \bar{\mb{F}}_p)^{K^p} _{\mf{n}}) \simeq H^{1}_c(Sh ^{an} _{K^p}, \bar{\mb{F}}_p) _{\mf{n}}$$
This can be possibly used to study the mod p cohomology of the Shimura curve $Sh$. We shall treat this issue elsewhere.

\section{Concluding remarks}

Let us finish by giving some remarks and stating natural questions.

\subsection{l-adic case} Observe that our arguments work well also in the mod $l\not =p$ setting and circumvent the use of vanishing cycles. The idea of localisation at a supersingular (supercuspidal) representation appears also in the work of Dat. See especially \cite{da} where the author discusses localisations both for $\GL_n$ and quaternion algebras and then uses it to describe the supercuspidal part of the cohomology. 

One might want also to see \cite{sh}, which bears some resemblance to certain arguments we use. Shin describes the mod l cohomology of Shimura varieties by using results of Dat about the mod l cohomology of the Lubin-Tate tower. In our work, we start from global results of Emerton to deduce from them statements about local objects.

\subsection{Beyond modular curves} The geometric arguments we have given also applies to Shimura curves considered by Carayol in \cite{ca} and we can consider similar exact sequences relating the ordinary locus and the supersingular locus in this setting. Nevertheless, in this case we cannot go on with arguments as we do not have a definition of the mod p local Langlands correspondence for extensions of $\mb{Q}_p$. In fact, such a construction seems a little bit problematic as might be seen from the work of Breuil-Paskunas (\cite{bp}), where the authors show that there are much more automorphic representations than Galois representations. The hope is that by looking at the cohomology of the Lubin-Tate tower, one should be able to tell how the correspondence should look like. We will pursue this subject in our subsequent work.

\subsection{Adic spaces} We have chosen to work with Berkovich spaces, but one might as well wonder how the things translate into the setting of adic spaces of R. Huber (\cite{hu}). In fact, everything that we have considered can be rewritten in the language of adic spaces and we might consider the same long exact sequences as above (though these exact sequences will be inversed due to the fact that adic spaces behave like formal schemes). The main difference between those two contexts lies in the ordinary locus which in the case of adic spaces will contain additional points which lie in the closure of the ordinary locus from the setting of Berkovich spaces. Nevertheless, the cohomology groups in both settings will be similar. Let us remark also, that the comparison between mod p \'etale cohomology of a formal scheme and its (adic) analytification is proved in Theorem 3.7.2 of \cite{hu}.

\subsection{Serre's letters} Though it does not appear explicitely in our work (besides the comparison of Hecke algebras), we were influenced by two letters written by Jean-Pierre Serre (see \cite{se}). It is there that in some sense appears for the first time the modified mod l Local Langlands correspondence which goes under the name of the universal unramified representation (see the letter to Kazhdan). Indeed, if we were to suppose that our global lift $\bar{\rho}$ which we have used is actually unramified everywhere outside $p$, then there is no need to recall either the modified mod l Local Langlands correspondence or new vectors, and we could formulate everything in the language of Serre.

\appendix
\section{The Hochschild-Serre spectral sequence}

In the body of the text we have used the unpublished manuscript of Berkovich (\cite{ber3}) where, among others, appears the Hochschild-Serre spectral sequence for the cohomology with compact support. For the sake of completeness, we will sketch a proof of existence of this spectral sequence here. We thank Vladimir Berkovich for sending us his preprint.

\subsection{G-spaces} Recall that an analytic space (in the sense of Berkovich) is a $k$-analytic space over some non-archimedean field $k$. Given two analytic spaces $X$ and $Y$, let $\Mor(X,Y)$ denote the set of morphisms $X \ra Y$ and let $\mc{G}(X)$ be the group of automorphisms of $X$. Berkovich defined in \cite{ber1} a uniform space structure (and in particular, a topology) on $\Mor(X,Y)$. Then, the group $\mc{G}(X)$ has the topology induced from $\Mor(X,X)$. We say that the action of a topological group $G$ on an analytic space $X$ is continuous if the induced homomorphism $G \ra \mc{G}(X)$ is continuous. An analytic space endowed with a continuous action of a topological group $G$ will be called a $G$-space. A $G$-equivariant morphism between two $G$-spaces will be called a $G$-morphism. The category of analytic spaces is the category of pairs $(X,G)$, where $G$ is a topological group and $X$ is a $G$-space. We will denote this pair by $X(G)$. A morphism between such spaces $\phi : X'(G') \ra X(G)$ is a pair consisting of a continuous homomorphism of topological groups $\nu _{\phi}: G' \ra G$ and a morphism of analytic spaces $\phi : X' \ra X$ compatible with the homomorphism $\nu _{\phi}$.  A $G$-morphism $\phi: X' \ra X$ between $G$-spaces gives rise to a morphism $\phi : X'(G) \ra X(G)$ for which $\nu _{\phi}$ is the identity map on $G$. If $X$ is a $G$-space then the action of $G$ on $X$ extends to a natural action of $G$ on $X(G)$ for which $\nu _{g} (g') = gg' g^{-1}$, where $\nu _g$ is the morphism given by an element $g \in G$. For a $G$-space $X$ we have a morphism $b: X \ra X(G)$ where $X = X(\{1\})$.

\subsection{\'Etale topology} Berkovich has defined the \'etale topology on analytic spaces and similarly we can define the \'etale topology on $G$-analytic spaces. For a $G$-space $X$, let $Et(X(G))$ denote the category of \'etale morphisms $U(G) \ra X(G)$. The \'etale topology on $X(G)$ is the Grothendieck topology on the category $Et(X(G))$ with coverings of $U(G)\ra X(G)$ consisting of families $(U_i(G) \ra U(G))_{i \in I}$ such that $(U_i \ra U)_{i \in I}$ is a covering in the \'etale topology of $X$. We denote this site by $X(G) _{et}$ and its corresponding topos by $X(G) _{et} ^{\sim}$. In a similar way, we can also introduce a quasi-\'etale site $X(G) _{qet}$ and its topos.

We denote by $\Gamma _{X(G)}$ the functor of global sections on $X(G) _{et} ^{~}$, that is $\Gamma _{X(G)}(F) = F(X(G))$. The higher direct images of $\Gamma _{X(G)}$ on the category of abelian sheaves will be denoted by $F \mapsto H^i(X(G),F)$. Let $F$ be an \'etale abelian sheaf on $X(G)$. The support of $f \in F(X(G))$ is the (closed) set $\Supp (f) = \{ x \in X \ | \ f_x \not =0 \}$, where $f_x$ is the image of $f$ in $F_x$. The cohomology groups with compact support are higher direct images of the functor $F \mapsto \Gamma _{c, X(G)}(F) = \{ f \in F(X(G)) \ | \ \Supp(f) \textrm{ is compact} \}$ and are denoted by $F \mapsto H^i_c(X(G),F)$. We consider also the higher direct functor of $F \mapsto \Gamma _{c, X\{G\}}(F) := \varinjlim \Gamma _{c, X(N)}(F)$ where $N$ runs through open subgroups of $G$ which we will denote by $F \mapsto H^i_c(X\{G\}, F)$. We have $H^i_c(X\{G\},F) = \varinjlim H^i _c (X(N), F)$. The proposition (Corollary 1.5.2 in \cite{ber3}) we have used in the main text was
\begin{prop} For any \'etale abelian sheaf $F$ on $X(G)$ there are canonical isomorphisms 
$$H^i_c(X\{G\},F) \simeq H^i _c(X, b^* F)$$
for $i \geq 0$. In particular, there is a spectral sequence
$$ E_2 ^{p,q} = H^p(G, H^q _c(X, b^* F)) \Rightarrow H^{p+q} _c(X(G),F)$$
\end{prop}
We have applied it to the Lubin-Tate tower $X = \varprojlim _n \mc{M} _{LT, n}$ and $G= I(1)$ the pro-p Iwahori subgroup of $\GL_2(\mb{Q}_p)$, by applying it to each $\mc{M} _{LT, n}$ and taking direct limit in the spectral sequence above (as $X$ is not a Berkovich space we cannot put it directly in the spectral sequence). We are using also the fact that we have an equivalence of topoi $X(G) _{et} ^{\sim} \simeq (G\backslash X) _{et} ^{\sim}$, whenever $G\backslash X$ exists and $X \ra G\backslash X$ is \'etale. This is so in our case, where $G\backslash X = \mc{M} _{LT, I(1)}$. 
\begin{proof} We sketch the proof of this proposition. The case $i = 0$ follows from the fact that every element of $H^0_c(X, b^*F)$ is fixed by an open subgroup of $G$. Then the general case follows by constructing the Godement resolution in our context. Namely, for a topological space $I$ denote by $Top(I)$ the site on the category of local homeomorphisms $J \ra I$ (with the evident topology). Suppose we have a surjective map $I \ra X : i \mapsto x_i$. We endow $I$ with the discrete topology and we fix for each $i \in I$ a geometric point $\bar{x}_i$ over $x_i$. This gives rise to a morphism of sites $\nu : Top(I) \ra X(G) _{et}$. For an \'etale abelian sheaf $F$ on $X(G)$, its Godement resolution $\mc{C} ^{\bullet}(F)$ is constructed as follows:

(i) $\mc{C} ^0(F) = \nu _* \nu ^* (F)$ and let $d^{-1}: F \ra \mc{C}^0 (F)$ be the adjunction map

(ii) if $m \geq 0$, then put $\mc{C} ^{m+1}(F) = \mc{C}^0(coker \ d^{m-1})$, and let $d^m$ be the canonical map $\mc{C}^m(F) \ra \mc{C}^{m+1}(F)$.

This construction is taken from SGA 4, Exp. XVII, where it is shown that

(a) $\mc{C} ^m(F)$ is a flabby sheaf

(b) the functor $F \mapsto \mc{C} ^m(F)$ is exact

(c) the fibre of the complex $\mc{C} ^{\bullet}(F)$ at a point $x \in X$ is a canonically split resolution of $F_x$. 

Then Berkovich shows (Proposition 1.5.1), that for any $F \in X(G)_{et} ^{\sim}$ and $m \geq 0$, the sheaf $b^*(\mc{C}^m(F))$ is soft on $X_{et}$, where we say that a sheaf $\mc{F}$ is soft on $X_{et}$ (after chapter 3 in \cite{ber1}) when it satisfies the following two conditions

(1) for any $x\in X$, the stalk $\mc{F} _x$ is a flabby $\Gal _{\mc{H}(x)} = \Gal (\overline{\mc{H}(x)} / \mc{H}(x))$-module, where $\mc{H}(x)$ is the complete field associated to $x$ by the standard procedure.

(2) for any paracompact $U$ \'etale over $X$, the restriction of $\mc{F}$ to the usual topology $|U|$ of $U$ is a soft sheaf, that is, for any compact subset $T \subset U$, the map $\mc{F}(U) \ra \mc{F}(T)$ is surjective.

Then one shows (see Lemma 3.2 in \cite{ber1}) that we can compute the cohomology with soft sheaves on $X_{et}$. Remark also that to prove the theorem it is enough to prove that $\mc{F} = b^*(\mc{C}^0(F))$ is soft, by the inductive definition of the Godement resolution. This is done by checking explicitely the conditions (1) and (2) for such an $\mc{F}$. We omit the computations.
\end{proof}


\begin{thebibliography}{9}

\bibitem[Ber1]{ber} V. Berkovich, "`On the comparison theorem for étale cohomology of non-Archimedean analytic spaces"' Israel J. Math. 92 (1995), 45-60

\bibitem[Ber2]{ber1} V. Berkovich "`Vanishing cycles for formal schemes I"', Invent. Math. 115 (1994), no. 3, 539-571

\bibitem[Ber3]{ber2} V. Berkovich, "`Vanishing cycles for formal schemes II"', Invent. Math. 125 (1996), no. 2, 367-390

\bibitem[Ber4]{ber3} V. Berkovich, "`Étale equivariant sheaves on p-adic analytic spaces"', preprint 1995

\bibitem[Ber5]{ber4} V. Berkovich, "`Étale cohomology for non-Archimedean analytic spaces"', Publ. Math. IHES 78 (1993), 5-161.

\bibitem[Be1]{be} L. Berger, "`Central characters for smooth irreducible modular representations of $\GL_2(\mb{Q} _p)$"', Rendiconti del Seminario Matematico della Universita di Padova (F. Baldassarri's 60th birthday)

\bibitem[Be2]{be2} L. Berger, "`La correspondence de Langlands locale p-adique pour $\GL_2(\mb{Q}_p)$"', seminaire Bourbaki
    Asterisque 339 (2011), 157-180

\bibitem[Bo]{bo} P. Boyer, "`Mauvaise reduction des varietes de Drinfeld et correspondance de Langlands locale"', Invent. Math. 138 pp 573-629 (1999)

\bibitem[BD]{bd} C. Breuil, F. Diamond, "`Formes modulaires de Hilbert modulo p et valeurs d'extensions galoisiennes"', preprint 2012

\bibitem[BP]{bp} C. Breuil, V. Paskunas, "`Towards a modulo p Langlands correspondence for $\GL_2$"', Memoirs of Amer. Math. Soc. 216, 2012.

\bibitem[Br1]{br1} C. Breuil, "`Sur quelques représentations modulaires et p-adiques de $\GL_2(\mb{Q} _p)$ I"', Compositio Math. 138, 2003, 165-188.

\bibitem[Br2]{br2} C. Breuil, "`Sur quelques représentations modulaires et p-adiques de $\GL_2(\mb{Q} _p)$ II"', J. Inst. Math. Jussieu 2, 2003, 1-36.

\bibitem[Br3]{br3} C. Breuil, "`Sur un probleme de compatibilite local-global modulo p pour $\GL_2$"', to appear in J. Reine Angew. Math.

\bibitem[Bu]{bu} K. Buzzard "`Analytic continuation of overconvergent eigenforms"', Journal of the American Math. Society 16 (2003), 29-55

\bibitem[BDJ]{bdj} K. Buzzard, F. Diamond, F. Jarvis, "On Serre's conjecture for mod l Galois representations over totally real fields", Duke Math Journal Vol 55 no. 1, 2010, 105-161.

\bibitem[BG]{bg} K. Buzzard, T. Gee, "`Explicit reduction modulo p of certain 2-dimensional crystalline representations II"`, to appear in Bulletin of the LMS.

\bibitem[Ca]{ca} H. Carayol "`Sur les representations l-adiques associees aux formes modulaires de Hilbert"', Ann. Sci. Ecole Norm. Sup. (4) 19 (1986), no. 3, 409-468

\bibitem[Cas]{cas} W. Casselman "`On some results of Atkin and Lehner"', Mathematische Annalen 201 (1973), 301-314

\bibitem[CE]{ce} F. Calegari, M. Emerton "`Completed cohomology - a survey"', to appear in Nonabelian Fundamental Groups and Iwasawa Theory, Cambridge University Press

\bibitem[Co]{co} P. Colmez, "`Representations de $\GL_2(\mb{Q}_p)$ et $(\phi,\Gamma)$-modules"', Asterisque 330 (2010), 281-509

\bibitem[Col]{col} R. Coleman, "`On the components of $X_0(p^n)$"', J. Number Theory 110 (2005), no. 1, 3-21

\bibitem[Da1]{da} J.-F. Dat "`Theorie de Lubin-Tate non-abelienne l-entiere"', Duke Math. J. 161 (6); 951-1010, (2012)

\bibitem[Da2]{da2} J.-F. Dat "`Espaces symetriques de Drinfeld et correspondance de Langlands locale"', Ann. Scient. Ec. Norm. Sup, 39 (1); 1-74 (2006)

\bibitem[De]{de} P. Deligne, letter to Piatetski-Shapiro, 1973

\bibitem[EH]{eh} M. Emerton, D. Helm, "`The local Langlands correspondence for $\GL_n$ in families"', preprint 2011

\bibitem[EGH]{egh} M. Emerton, T. Gee, F. Herzig, "`Weight cycling and Serre-type conjectures for unitary groups"', preprint 2011

\bibitem[Em1]{em1} M. Emerton "`On the interpolation of systems of Hecke eigenvalues"', Invent. Math. 164 (2006), no. 1, 1-84

\bibitem[Em2]{em2} M. Emerton "`Local-global compatibility in the p-adic Langlands programme for $\GL_2(\mb{Q})$"', preprint 2011

\bibitem[Em3]{em3} M. Emerton "`Ordinary parts of admissible representations of p-adic reductive groups I. Definition and first properties"' Asterisque 331 (2010), 335-381.

\bibitem[EP]{ep} M. Emerton, V. Paskunas, "`On effaceability of certain $\delta$-functors"', Asterisque 331 (2010), 439-447.

\bibitem[Fa1]{fa} L. Fargues "`Cohomologie des espaces de modules de groupes p-divisibles et correspondances de Langlands locales"' Asterisque 291 (2004) p. 1-200

\bibitem[Fa2]{fa2} L. Fargues "`L'isomorphisme entres les tours de Lubin-Tate et de Drinfeld et applications cohomologiques"', in "`L'isomorphisme entre les tours de Lubin-Tate et de Drinfeld"' Birkhauser, Progress in Math, vol. 262 

\bibitem[GS]{gs} T. Gee, D. Savitt "`Serre weights for quaternion algebras"', Compositio Mathematica Volume 147, Issue 04, 1059-1086 (2011).

\bibitem[HR]{hr} T. Haines, M. Rapoport, "`Shimura varieties with $\Gamma _1(p)$-level via Hecke algebra isomorphisms: The Drinfeld case"', preprint 2010

\bibitem[HT]{ht} M. Harris, R. Taylor "`The geometry and cohomology of some simple Shimura varieties"' Annals of Math. Studies 151, PUP 2001.

\bibitem[He]{he} D. Helm, "`On the modified mod p Local Langlands correspondence for $\GL _2(\mb{Q}_l)$"', preprint 2012

\bibitem[Hu]{hu} R. Huber, "` \'Etale cohomology of rigid analytic varieties and adic spaces"', Aspects of Mathematics,
E30, Friedr. Vieweg and Sohn, Braunschweig, 1996

\bibitem[Kh1]{kh} C. Khare "`Serre's modularity conjecture: the level one case"', Duke Math. J. 134 (2006), no. 3, 557-589

\bibitem[Kh2]{kh2} C. Khare "`A local analysis of congruences in the (p,p) case: Part II"' Invent. Math. 143, 129-155 (2001)

\bibitem[KM]{km} N. Katz, B. Mazur "`Arithmetic moduli of eliptic curves"', Annals of Mathematics Studies, 108. Princeton University Press, Princeton, NJ, 1985

\bibitem[Pa1]{pa1} V. Paskunas "`On the image of Colmez's Montreal functor"', preprint 2011

\bibitem[Pa2]{pa2} V. Paskunas "`Extensions for supersingular representations of $\GL_2(\mb{Q} _p)$"', Asterisque 331 (2010) 297-333

\bibitem[Pa3]{pa3} V. Paskunas "`Blocks for mod p representations of $\GL_2(\mb{Q}_p)$"', preprint 2011

\bibitem[Pa4]{pa4} V. Paskunas "`On the restriction of representations of $\GL_2(F)$ to a Borel subgroup"', Compos. Math. 143 (2007) no. 6, 1533-1544

\bibitem[Pa5]{pa5} V. Paskunas "`Coefficient systems and supersingular representations of $\GL _2(F)$"', Mem. Soc. Math. Fr. (N.S.) 99 (2004) 

\bibitem[Pi]{pi} R. Pink, "Arithmetical compactification of mixed Shimura varieties", Bonner Math. Schriften, 209 (1990).

\bibitem[Se]{se} J.-P. Serre "`Two letters on quaternions and modular forms (mod p)"', Israel J. Math. 95 (1996), 281-299

\bibitem[Sh]{sh} S.W. Shin "`Supercuspidal part of the mod l cohomology of GU(1,n-1)-Shimura varieties"', preprint 2011

\bibitem[Str]{str} M. Strauch, "`Geometrically connected components of deformation spaces of one-dimensional formal modules"',
Pure and Applied Mathematics Quarterly, vol. 4, Number 4, p. 1-18 (2008).

\bibitem[Vi1]{vi} M.-F. Vigneras, "`Representations modulaires galois-quaternions pour un corps p-adique"', Journees Arithmétiques d'Ulm, Springer-Verlag Lecture Notes 1380, page 254-266 (1989) 

\bibitem[Vi2]{vi2} M.-F. Vigneras, "`Representations modulaires $\GL(2,F)$ en caracteristique $l$, $F$ corps p-adique, $p \neq l$"' Compositio Mathematica (1989).

\bibitem[Vi3]{vi3} M.-F. Vigneras, "`Representations modulo p of the p-adic group $\GL(2,F)$"',  Compositio Math. 140 (2004) 333-358

\bibitem[Yo]{yo} T. Yoshida, "`On non-abelian Lubin-Tate theory via vanishing cycles"', Advanced Studies in Pure Mathematics, 58 (2010), 361-402

\end{thebibliography}
\end{document}